\documentclass[11pt]{amsart}

\usepackage{amsmath,amssymb,amscd,amsthm, url}
\usepackage[all]{xy}
\usepackage{enumerate, fullpage}
\usepackage{mathtools}
\usepackage[mathscr]{euscript}
\usepackage[dvipdfmx]{hyperref}
\usepackage[foot]{amsaddr}
\theoremstyle{plain}
\numberwithin{equation}{section}
\newtheorem{thm}{Theorem}[section]
\newtheorem*{thm*}{Theorem}
\newtheorem{prop}[thm]{Proposition}
\newtheorem{prop-dfn}[thm]{Proposition-Definition}
\newtheorem{prop'}[thm]{``Proposition"}
\newtheorem{cor}[thm]{Corollary}
\newtheorem{lem}[thm]{Lemma}

\theoremstyle{definition}
\newtheorem{dfn}[thm]{Definition}
\newtheorem{note}[thm]{Notation}

\newtheorem{rmk}[thm]{Remark}
\newtheorem{prb}[thm]{Problem}

\def\rank{\mathop{\mathrm{rank}}\nolimits}
\def\dim{\mathop{\mathrm{dim}}\nolimits}
\def\Im{\mathop{\mathrm{Im}}\nolimits}

\def\Hom{\mathop{\mathrm{Hom}}\nolimits}

\def\fib{\mathop{\mathsf{fib}}\nolimits}
\def\cof{\mathop{\mathsf{cof}}\nolimits}

\def\ker{\mathop{\mathsf{ker}}\nolimits}
\def\cok{\mathop{\mathsf{cok}}\nolimits}
\def\im{\mathop{\mathsf{im}}\nolimits}

\def\K{{\mathit{Kom}}}
\def\Map{{\mathrm{Map}}}

\def\<{{\langle}}
\def\>{{\rangle}}

\def\+{\mathop{\oplus}\nolimits}
\def\1{\mathop{\mathrm{id}}\nolimits}
\def\Spec{\mathop{\mathrm{Spec}}\nolimits}

\newcommand{\Aut}[1]{{\mathrm{Aut}\,{#1}}}
\newcommand{\Dom}[1]{{\mathrm{Dom}({#1}^{-1})}}
\newcommand{\gl}[2]{{\mathsf{gl}\left({#1} ,  {#2}\right)}}
\newcommand{\Stab}[1]{{\mathsf{Stab}\,{#1}}}
\newcommand{\Stabd}[1]{{\mathsf{Stab}^{\dagger}{#1}}}
\newcommand{\Stabr}[1]{{\mathsf{Stab}^{\mathsf{r}}\,{#1}}}

\newcommand{\Fun}[2]{{\mr{Fun}( {#1}, {#2})}}
\newcommand{\ho}[1]{{\mr{h}{(#1)}}}

\newcommand{\pair}[2]{{\left( {#1},{#2} \right)}}
\newcommand{\HNF}[3]{{
\xymatrix{
0	\ar[r]	&	{#2}_1	\ar[r]\ar[d]	&	{#2}_2	\ar[r]\ar[d]	&	\cdots\ar[r]	&	{#2}_{n-1}	\ar[r]\ar[d]	&	{#2}_n={#1}\ar[d]	\\
			&	{#3}_1\ar@{-->}[ul]	&	{#3}_2\ar@{-->}[ul] &					&	{#3}_{n-1}\ar@{-->}[ul]		&	{#3}_n .\ar@{-->}[ul]	
}
}}

\newcommand{\bb}[1]{{\mathbb{#1}}}
\newcommand{\mca}[1]{{\mathcal{#1}}}
\newcommand{\mr}[1]{{\mathrm{#1}}}
\newcommand{\ms}[1]{{\mathscr{#1}}}
\newcommand{\mb}[1]{{\mathbf{#1}}}

\title[Stability conditions on morphisms in a category]{Stability conditions on morphisms in a category}
\author{Kotaro Kawatani}
\email{kawatanikotaro@gmail.com}
\address{Yamato University, Katayama-cho, Suita-city, Osaka, Japan. Osaka }
\address{Osaka Prefecture University, Gakuen-cho, Sakai-city, Osaka, Japan. Osaka }
\keywords{Triangulated categories, Stability conditions, Functor category, Stable infinity category}
\date{\today}

\subjclass[2010]{18E30}

\begin{document}
\maketitle

\begin{abstract}
Let $\mathrm{h}\mathscr{C}$ be the homotopy category of a stable infinity category $\mathscr{C}$. 
Then the homotopy category $\mathrm{h}\mathscr{C}^{\Delta^{1}}$ of morphisms in the stable infinity category $\mathscr{C}$ is also triangulated. 
Hence the space $\mathsf{Stab}\,{ \mathrm{h}\mathscr{C}^{\Delta^{1}}}$ of stability conditions on $\mathrm{h}\mathscr{C}^{\Delta^{1}}$ is well-defined though the non-emptiness of $\mathsf{Stab}\,{ \mathrm{h}\mathscr{C}^{\Delta^{1}}}$ is not obvious. 
Our basic motivation is a comparison of the homotopy type of $\Stab{\mathrm{h}\mathscr{C}}$ and that of $\Stab{\mathrm{h}\mathscr{C}^{\Delta^{1}}}$. 
Under the motivation we show that functors $d_{0}$ and $d_{1} \colon \mathscr{C}^{\Delta^{1}} \rightrightarrows  \mathscr{C}$ induce continuous maps from $\Stab {\mathrm{h}\mathscr{C}}$ to $\Stab{\mathrm{h}\mathscr{C}^{\Delta^{1}}}$ contravariantly where $d_{0}$ (resp. $d_{1}$) takes a morphism to the target (resp. source) of the morphism. 
As a consequence, if $\Stab{\mathrm{h}\mathscr{C}}$ is nonempty then so is $\Stab{\mathrm{h}\mathscr{C}^{\Delta^{1}}}$. 
Assuming $\mathscr{C}$ is the derived infinity category of the projective line over a field, we further study basic properties of $d_{0}^{*} $ and $d_{1}^{*}$. 
In addition, we give an example of a derived category which does not have any stability condition. 
\end{abstract}

\section{Introduction}

Let $\mb D$ be a triangulated category. 
The space $\Stab {\mb D}$ of locally finite stability conditions on $\mb D$ was introduced by Bridgeland \cite{MR2373143}. 
One of remarkable features of $\Stab {\mb D}$ is that each connected component of $\Stab {\mb D}$ is a complex manifold  unless $\Stab{\mb D}$ is empty\footnote{In this article, an empty set is not a complex manifold. }. 
In general the non-emptiness of $\Stab {\mb D}$ is not obvious and the connectedness is an open problem. 
For instance, if $\mb D$ is the bounded derived category $\mb D^{b}(\mb{coh}\, \bb A^{1}_{\mb k})$ of coherent sheaves on the affine line over a field $\mb k$, 
$\Stab{\mb D^{b}(\mb{coh}\, \bb A^{1}_{\mb k})}$ is empty by Proposition \ref{pr:empty}.

If a scheme is projective, there are many non-empty examples. 
Let $\mb D^{b}(\mb{coh}\, M)$ be the bounded derived category of coherent sheaves on a smooth projective variety $M$. 
If $\dim M=1$, the space $\Stab {\mb D^{b}(\mb{coh}\, M)}$ of stability conditions on $\mb D^{b}(\mb{coh}\, M)$ is not empty and is connected by \cite{MR2373143}, \cite{MR2335991}, and \cite{MR2219846}. 
If $\dim M=2$, $\Stab{\mb D^{b}(\mb{coh}\, M)}$ is non-empty by \cite{MR2998828} and the connectedness is open.  
If $\dim M=3$, Bayer-Macri-Toda \cite{MR3121850} shows that a generalized Bogomolov-Gieseker type inequality implies the non-emptiness of $\Stab{\mb D^{b}(\mb{coh}\,M)}$.

It is very difficult to describe $\Stab {\mb D}$ globally. 
Concerned with the description the following working hypothesis states that $\Stab {\mb D}$ should be globally simple in a homotopical view: 
\begin{description}
\item[Hypothesis]The space $\Stab {\mb D}$ of stability conditions on $\mb D$ is contractible unless $\Stab{\mb D}$ is empty. 
\end{description}
If $\mb D  = \mb D^{b}(\mb{coh}\, M)$ with $\dim M=1$, then the hypothesis holds by \cite{MR2373143}, \cite{MR2335991} and \cite{MR2219846}.  Moreover there are no counter-example to the best of my knowledge. 

Specializing $\mb D$, we can specify the origin of the hypothesis. 
Let $X$ be the minimal resolution of a Kleinian singularity and $Z$ the schematic fiber of the singularity. 
Suppose that $\mb D$ is the category $\mb D_{Z}^{b}(\mb{coh}\, X)$ spanned by bounded complexes in $\mb D^{b}(\mb{coh}\, X)$ supported in $Z$. 
The space $\Stab{\mb D_{Z}^{b}(\mb{coh}\, X)}$ is conjecturally\footnote{If the singularity is type $A_{n}$ then $\Stab{\mb D^{b}_{Z}(\mb{coh}\, X)}$ is actually the universal cover by \cite{MR2629510}. } the universal covering space over a certain configuration space (see also \cite{MR2376815} and \cite{MR2549952}). 
It seems natural to expect that $\Stab{\mb D}$ is contractible in general since the configuration space should be an Eilenberg MacLane space by the $K(\pi,1)$ conjecture deriving from Brieskron \cite{MR0422674}  and Arnol'd (see also \cite{MR3205598}).

Stimulated by the hypothesis, it would be interesting to study the homotopy type of $\Stab{\mb D}$. 
Our basic motivation is 
a comparison of homotopy types of the spaces of stability conditions on $\mb D$ and on the category $\mr{Mor}(\mb D)$ of morphisms in $\mb D$. 
Unfortunately the category $\mr{Mor}(\mb D)$ is not triangulated in general, 
so we start with a stable infinity category (see also \S \ref{sc:stable}) which is a candidate of ``enhanced'' triangulated categories.

Let $\ms C$ be a stable infinity category. 
Then the infinity category $\ms C^{\Delta^{1}}$ of morphisms in $\ms C$ is also stable. 
Since the homotopy category of a stable infinity category is triangulated, we could introduce a triangulated structure on $\ho{\ms C^{\Delta^{1}}}$. 
Thus the space of stability conditions on $\ho {\ms C^{\Delta^1}}$ is well-defined though the non-emptiness of $\Stab {\ho {\ms C^{\Delta^{1}}}}$ is not obvious. 
The following is a natural question motivated by the hypothesis: 
 
\begin{prb}\label{problem1}
Is $\Stab {\ho {\ms C^{\Delta^{1}}}}$ homotopy equivalent to $\Stab {\ho {\ms C}}$?
\end{prb}
If the answer of Problem \ref{problem1} is negative we might find an interesting counter-example of the hypothesis. 
We note that the answer to the easiest case of the problem is affirmative. 
Precisely if $\ms C$ is the infinity category $\ms D_{\mr{coh}}^{b}(\Spec \mb k)$ of $\Spec \mb k$ (see Definition \ref{dfn:derivedmugen}) then $\Stab {\ho {\ms C^{\Delta^{1}}}}$ is homotopy equivalent to 
$\Stab{ \ho {\ms C} }$ by Proposition \ref{pr:easiest}.

It is difficult to generalize the argument of the easiest case since the answer comes from calculating $\Stab {\ho {\ms C}}$ and $\Stab {\ho {\ms C^{\Delta^{1}}}}$ independently.  
Our aim is a construction of maps between $\Stab {\ho {\ms C}}$ and $\Stab {\ho {\ms C^{\Delta^{1}}}}$ to study Problem \ref{problem1} in more general cases. 

Before main theorem, let us recall that there exist functors 
$d_0$ and $d_1 \colon \xymatrix{ \ms C^{\Delta^{1}}	\ar@<0.6ex>[r] \ar@<-0.6ex>[r] 	&  \ms C	}$ which take a morphism $f$ in $\ms C$ to the target and to the source of $f$ respectively. 
Though an exact functor between triangulated categories does not induce a map between the spaces of stability conditions in general, we show that both functors $d_0$ and $d_1$ respectively induce continuous maps $d_0^*$ and $d_1^*$ contravariantly. 
 The following theorem states some basic properties of these continuous maps. 

\begin{thm}\label{thm-main1}
Let $\ms  C$ be  a stable infinity category. 
Assume that the rank of the Grothendieck group $K_0(\ho{\ms C})$ is finite. 
\begin{enumerate}
\item \label{main1-1}If there exists a reasonable stability condition on $\ho{\ms C}$ then there exists a reasonable stability condition on $\ho{\ms C^{\Delta^{1}}}$. 
Moreover both functors $d_0$ and $d_1$ induce continuous and injective maps $d_0^*$ and $d_1^*$ from the space $\Stabr {\ho {\ms C}}$ of reasonable stability conditions on $\ho{\ms C}$ to that of $\ho {\ms C^{\Delta^{1}}}$: 
\[
d_0^*, d_1^* \colon \xymatrix{ \Stabr {\ho{\ms C}}	\ar@<0.6ex>[r] \ar@<-0.6ex>[r] 	& \Stabr {\ho{\ms C^{\Delta^{1}}}}	}. 
\]
\item \label{main1-2}The images $\Im d_0^*$ and $\Im d_1^*$ do not intersect each other. 
\item \label{main1-3}Both images $\Im d_0^*$ and $\Im d_1^*$ are closed in $\Stabr {\ho{\ms C^{\Delta^{1}}}}$. 
\item \label{main1-4}A stability condition $\sigma$ is full if and only if $d_0^* \sigma$ (or $d_1^* \sigma$) is full. 
\end{enumerate}
\end{thm}

Collins-Polishchuk \cite{MR2721656} constructed a ``glued" stability condition from a semiorthogonal decomposition of a triangulated category. 
Since the categoy $\ho{\ms C^{\Delta^{1}}}$ of morphisms has semiorthogonal decompositions with $\ho{\ms C}$ and $\ho{\ms C}$ (the details are in \S \ref{sc:SOD}), the gluing construction is effecitve in Theorem \ref{thm-main1}. 
Reasonable stability conditions (see Definition \ref{dfn-reasonable}) are necessary for the continuity of $d_{0}^{*}$ and $d_{1}^{*}$. 
In addition full stability conditions are most basic stability conditions (see also Section \ref{supportproperty}), and we do not know whether there exists a non full stability condition does exists or not. 
Since a full stability condition is reasonable, reasonable stability conditions are sufficiently ``reasonable".

Assertions (1) and (2) are consequences of the gluing construction. 
We use an ``inducing construction'' developed in \cite{MR2524593} to prove the third assertion (3) in Theorem \ref{thm-main1}. 
Roughly speaking the inducing construction asserts that a faithfull exact functor $F \colon \mb D \to \mb D'$ between triangulated categories induces a continuous map from a closed subset of $\Stab {\mb D'}$ to $\Stab{\mb D}$.   
The fourth assertion (4) follows from a necessary and sufficient condition for semistable objects in $\ho{\ms C^{\Delta^{1}}}$ with respect to $d_0^* \sigma$ and $d_1^* \sigma$.

The following problem is derived from the fourth assertion in Theorem \ref{thm-main1}:

\begin{prb}\label{problem2}
Is the image $\Im d_0^*$ path connected to $\Im d_1^*$? 
\end{prb}

Unfortunately it is quite difficult to prove the connectedness of the space of stability conditions and there are no counter-example such that the subspace of full stability conditions is connected. 
Thus it is natural to expect that both images $\Im d_0^*$ and $\Im d_1^*$ are path connected and the following gives an evidence. 

\begin{thm}\label{thm-main2}
If $\ms C$ is the infinity category $\ms D_{\mr{coh}}^{b}(\bb P^{1})$ of projective line $\bb P^1$ over a field $\mb k$, then the image $\Im d_0^*$ is path connected to $\Im d_1^*$. 
\end{thm}

In the proof of Theorem \ref{thm-main2} above we use an algebraic stability condition which does not exist in $\Stab {\ho{\ms D_{\mr{coh}}^{b}(C)}	}$ when the genus of the curve $C$ is grater than $0$. 
It is natural to study Theorem \ref{thm-main2} for positive genus cases. 
The case of $g(C)=1$ are discussed in \cite{2019arXiv190504240M} independently and their result gives an affirmative answer to Problem \ref{problem2}.

We note that the same argument in Theorem \ref{thm-main2} is effective for the full component $\Stabd {	\ho{\ms D^{b}_{\mr{coh}}(\bb P^{2})}	}$ of the space of stability conditions on the projective plane $\bb P^2$. 
Thus Theorem \ref{thm-main2} gives an evidence of Problem \ref{problem2}.

The organization of this article is the following. 
Section 2 is basically a summary on the infinity category of morphisms. 
We also describe the Serre functor on $\ho {\ms C^{\Delta^{1}}}$ under some assumptions on $\ms C$. 
The description gives an answer to \cite[Conjecture 3.17]{2019arXiv190504240M} (see also Remark \ref{rmk:relevant}). 
In Section 3 we observe Problem \ref{problem1} in the easiest case. 
In the observation, we show that the space $\Stab{\mb D^{b}(\mb{mod}\, \mb k)}$ 
for arbitrary field $\mb k$ is isomorphic to $\bb C$ (Corollary \ref{cor:speck}) where $\mb {mod}\, \mb k$ is the category of finite dimensional $\mb k$-spaces.  
One of key ingredients is the indecomposability of $\sigma $-stable objects. 
As an application, we show that the bounded derived category of finite generated modules on a principle ideal domain has no stability condition (Proposition \ref{pr:empty}). 

The first and second assertion in Theorem \ref{thm-main1} is proven in Propositions \ref{pr:pullback} and \ref{pr:fundamental}. 
The third assertion of Theorem \ref{thm-main1} is Theorem \ref{thm:closed}. 
The fourth assertion of Theorem \ref{thm-main1} is Theorem \ref{thm:full}. 
The section $5$ is devoted to prove Theorem \ref{thm-main2} which is just Theorem \ref{thm:P^1}.

In the last section, we discuss another construction of $\ho{\ms C^{\Delta^{1}}}$ when $\ms C$ is the infinity category $\ms D_{\mr{coh}}^{b}(X)$ for a Noetherian scheme $X$. 
Since the homotopy category $\ho{\ms D_{\mr{coh}}^{b}(X)}$ is equivalent to the bounded derived category $\mb D^{b}(\mb{coh}\, X)$, 
there is a natural bounded $t$-structure on $\ho{\ms D_{\mr{coh}}^{b}(X)}$ whose heart is the abelian category $\mca B= \mb{coh}(X)$ of coherent sheaves on $X$.  
Then the category $\mr{Mor}(\mca B)$ of morphisms in $\mca B$ is also an abelian category. 
Thus one could define the bounded derived category $\mb D^{b}\left(\mr{Mor}(\mca B) \right)$ via localization of quasi-isomorphisms. 
Corollary \ref{pr:MMR} claims that two categories $\ho{\ms D_{\mr{coh}}^{b}(X)^{\Delta^{1}}}$ and $\mb D^{b}\left(\mr{Mor}(\mca B) \right)$ are equivalent. 
As a consequence, $\ho{\ms D_{\mr{coh}}^{b}(\Spec \mb k)^{\Delta^{1}}}$ is equivalent to the bounded derived category of $A_{2}$ -quiver (see also Remark \ref{rmk:relevant}).

\subsection*{Acknowledgement}
I would like to thank H. Minamoto for his encouragement and thank a private communication with him which was helpful especially for  Proposition \ref{pr:eqF}. 
I also thank my family for their great support.  
Finally I am extremely grateful to the referee who patiently read this article, pointed out many inaccuracies, and relaxed the assumption of Lemma \ref{lm:indecomposable} in the earlier version  to additive categories.

\section{A triangulated category of morphisms in $\ms C$}

Let $\mr{Mor}(\mb D)$ be the category of morphisms in a category $\mb D$. 
Unfortunately $\mr{Mor}(\mb D)$ is not a triangulated category in general when $\mb D$ is a triangulated category. 
To solve the problem, we start with a stable infinity category. 

\subsection{Stable infinity categories}\label{sc:stable}
Let $[n]=\{0 < 1< \cdots < n\}$ be the finite ordinal ($n\geq 0$) and let $\Delta$ be the category of finite ordinals. 
The $i$-th face map is denote by $d^{i} \colon [n] \to [n+1] $ ($0 \leq i \leq n+1$) and the $j$-th degeneracy map is denoted by $s^{j} \colon [n+1] \to [n]$ ($0 \leq j \leq n$). 
The standard $n$-simplex is denoted by $\Delta^{n}$. 
Given a simplicial set $K$, the set of $n$-simplices is denoted by $K_{n}$. 
We denote by $d_{i} \colon K_{n+1}  \to K_{n}$ (resp. $s_{j} \colon K_{n} \to K_{n+1}$) the map induced from $d^{i}$ (resp. $s^{j}$). 

Let $\ms C$ be an infinity category in the sense of \cite[Definition 1.1.2.4]{MR2522659}. 
Namely, an infinity category is a simplicial set $\ms C$ satisfying the following lifting property:
\begin{itemize}
\item For any inclusion $\xymatrix{\iota \colon \Lambda _{k}^{n } \ar@{^{(}->}[r]& \Delta^{n} }$ of the $k$-th horn 
where $0 < k < n$ and for any morphism $p \colon \Lambda _{k}^{n} \to \ms C$, there exists a morphism $\tilde p \colon \Delta^{n} \to \ms C$ such that $\tilde p\circ \iota =p$. 
\end{itemize}
 
A $0$-simplex of the simplicial set $\ms C$ is called an \textit{object} of the infinity category $\ms C$ and 
a $1$-simplex of $\ms C$ is called a \textit{morphism} in $\ms C$. 
The homotopy category $\ho{\ms C}$ of $\ms C$ is a usual category  whose object is the same as that of $\ms C$ 
and whose morphism is an equivalence class $[f]$ of a morphism $f$ in $\ms C$ (the details are in \cite[\S 1.2.3]{MR2522659}). 
One of descriptions of the equivalence relation is the following: 

\begin{prop-dfn}[{\cite[\S 1.2.3]{MR2522659}}]\label{dfn:homotpy}
Let $\ms C$ be an infinity category. 
\begin{enumerate}
\item Two morphisms $f$ and $g$ from $x$ to $ y$ in an infinity category $\ms C$ is homotopic if there exists a $2$-simplex 
$h \colon \Delta^{2} \to \ms C$ satisfying $d_{0}h = \1 _{y}, d_{1}h= g$ and $d_{2}h =f$. 
Then the relation of homotopy is an equivalence relation. 
\item 
The homotopy category $\ho{\ms C}$ has the same objects of $\ms C$. 
A morphism in $\ho{\ms C}$ is given by the equivalence class of $1$-simplices with respect to the relation defined by $(1)$. 
\end{enumerate}
\end{prop-dfn}

A functor $F \colon \ms C \to \ms D$ between infinity categories is nothing but a morphism as simplicial sets. 
The functor $F$ naturally induces the functor $\ho{\ms C} \to \ho{\ms D}$ between homotopy categories. 
Though the induced functor should be written as $\ho{F}$, we write simply $F$ by abusing notation.

An infinity category $\ms C$ is said to be \textit{stable} if $\ms C$ has a zero object, admits finite limits and colimits, and pushout squares coincide with pullback squares. 
If $\ms C$ is stable then $\ho{ \ms C}$ is triangulated by \cite[Theorem 1.1.2.14]{higheralgebra}. 
A zero object in a stable infinity category is denoted by $0$. 
The mapping cone of a morphism $f \colon x \to y $ in the triangulated category $ \ho {\ms C}$ is given by the pushout\footnote{The push out is calculated in the infinity category $\ms C$, not in the triangulated category $\ho{\ms C}$. } $0 \sqcup _{x} y$ of $f \colon x \to y$ in the infinity category $\ms C$. 

Before examples of stable infinity categories, we fix the following notation for differential graded categories (shortly dg categories): 

\begin{note}\label{note:notation}
Let $\ms D$ be a dg category over a commutative ring. 
\begin{itemize}
\item The cochain complex of morphisms in $\ms D$ is denoted by $\Map_{\ms D}(x,y)$ for $x $ and $y \in \ms D$. 
\item The degree $p$-th part of $\Map_{\ms D}(x,y)$ is denoted by $\Map_{\ms D}^{p}(x,y)$. 
\item The differential of $\Map_{\ms D}(x,y)$ is denoted by $\delta^{p} \colon \Map_{\ms D}^{p}(x,y) \to \Map_{\ms D}^{p+1}(x,y)$ or shortly $\delta^{p}$.  
\item The kernel of $\delta^{p} \colon \Map_{\ms D}^{p}(x,y) \to \Map_{\ms D}^{p+1}(x,y)$ is denoted by $Z^{p}_{\ms D}(x,y)$. 
\item If $\mca A$ is an additive category, the dg-category of cochain complexes in $\mca A$ is denoted by $\mr{Ch}(\mca A)$. 
\item The category of cochain complexes in the additive category $\mca A$ is denoted by $\K(\mca A)$. 
\end{itemize}
Note that the objects in $\mr{Ch}(\mca A)$ and $\K(\mca A)$ are the same, and we have $\Hom_{\K(\mca A)}(x,y)=Z^{0}_{\mr{Ch}(\mca A)}(x,y)$
\end{note}

One of good examples of stable infinity categories is a ``derived infinity category'' $\ms D(\mca A)$ of a Grothendieck abelian category $\mca A$. 
The following is necessary for a relation with $\ms D(A)$ and the dg category $\mr{Ch}(\mca A)$.  

\begin{prop-dfn}[{\cite[\S 1.3.1, \S 1.3.2]{higheralgebra}}]\label{dfn:dg-nerve}
Let $\ms D$ be a dg category over a commutative ring. 
Define a simplicial set $\mr{N}_{\mr {dg}}(\ms D)$ whose $n$-simplex $\sigma \in \mr{N}_{\mr {dg}}(\ms D)_{n}$ is the set of ordered pairs $(\{x_{i}\}_{i=0}^{n}, \{f_{I}\})$ where 
\begin{itemize}
\item Each $x_{i}$ is an object in $\ms D$. 
\item For any subset $I=\{	i_{-}< i_{m} < \cdots < i_{1}< i_{+}\} $ of the ordinal $[n]$ where $m\geq 0$, $f_{I}$ is an element in $\Map_{\ms D}^{-m}(x_{i_{-}}, x_{i_{+}})$ satisfying 
\[
\delta f_{I} = \begin{cases}
\sum_{1\leq j \leq m }(-1)^{j} \left(f_{I-\{i_{j}\}} - f_{ \{ i_{j}< \cdots < i_{1}<i_{+} \} } \circ f_{\{ i_{-}< i_{m} < \cdots <i_{j} \} }\right)	& (m >0)\\
0 &	(m=0). 
\end{cases}
\]
For a nondecreasing map $ \alpha \colon [m] \to [n]$, the induced map 
$\alpha ^{*} \colon \mr{N}_{\mr {dg}}(\ms D)_{n} \to \mr{N}_{\mr {dg}}(\ms D)_{m}$ is given by $(\{x_{i}\} _{0 \leq i \leq n}, \{f_{I}\}) \mapsto  (\{x_{\alpha(j)}\}_{0 \leq j \leq m} , \{g_{J}\})$ where 
\[
g_{J}= \begin{cases}
f_{\alpha(J)}	&	\mbox{if $\alpha|J$ is injective} \\
\1_{x_{i}}	& \mbox{if $J=\{j, j'\}$ with $\alpha (j)=\alpha(j')=i$} \\
0 & \mbox{otherwise}. 
\end{cases}
\]
\end{itemize}
Then the simplicial set $\mr{N}_{\mr {dg}}(\ms D)$ is an infinity category. 
In particular, the simplicial set $\mr N_{\mr {dg}}(\ms D)$ is referred as the differential graded nerve of $\ms D$. 
\end{prop-dfn}

\begin{rmk}\label{rmk:mugen-dg}
Keep the notation as in Proposition-Definition \ref{dfn:dg-nerve}. 
From the definition, the set $\mr {N}_{\mr {dg}}(\ms D)_{1}$ of $1$-simplices of $\mr {N}_{\mr {dg}}(\ms D)$ is  
\[
\{ f \in Z^{0}_{\ms D}(x,y) \mid x, y \in \ms D\}. 
\]
Similarly the set $\mr {N}_{\mr {dg}}(\ms D)_{2}$ is  
\[
\{	(f_{0}, f_{1}, f_{2}, h) \in Z^{0}_{\ms D}(y,z) \times Z^{0}_{\ms D}(x,z) \times Z^{0}_{\ms D}(x,y) \times \Map_{\ms D}^{-1}(x,z) \mid 
\delta (h) = f_{0}  f_{2}-f_{1}\mbox{ and }x,y,z \in \ms D	\}. 
\]
Thus a $2$-simplex of $\ms C= \mr{N}_{\mr {dg}}(\ms D)$ could be regarded as a $4$-tuple of morphisms in $\ms D$. 
\end{rmk}

\begin{prop-dfn}[{\cite[Proposition 1.3.5.3]{higheralgebra}}]\label{dfn:injetivemodel}
Suppose $\mca A$ is a Grothendieck abelian category. 
$\K(\mca A)$ has a (left proper combinatorial) model structure described as follows: 
\begin{itemize}
\item[(\textsf{w})] A map $\{f^{p} \colon E^{p} \to F^{p}\}_{p\in \bb Z}$ in $\K(\mca A)$ is a weak equivalence if $f$ is a quasi-isomorphism. 
 \item[(\textsf{c})] A map $\{f^{p} \colon E^{p} \to F^{p}\}_{p\in \bb Z}$ is a cofibration if $f^{p} \colon E^{p}\to F^{p}$ is a monomorphism for each degree $p \in \bb Z$. 
 \item[(\textsf{f})] A map $\{f^{p} \colon E^{p} \to F^{p}\}_{p\in \bb Z}$ is a fibration if it has the right lifting property with respect to every map which is a cofibration and a weak equivalence. 
\end{itemize}
We refer to the model structure as the injective model structure on $\K(\mca A)$. 
\end{prop-dfn}

\begin{rmk}
Any object in $\K(\mca A)$ is cofibrant with respect to the injective model structure. 
\end{rmk}

\begin{dfn}\label{dfn:derivedmugen}
Let $\mca A$ be a Grothendieck abelian category. 
\begin{enumerate}
\item Let $\mr{Ch}(\mca A)^{\circ }$ be the full sub-dg-category of $\mr{Ch}(\mca A)$ consisting of fibrant-cofibratn objects in $\K(\mca A)$ with respect to the injective model structure. 
Define the infinity category $\ms D(\mca A)$ by the infinity category $\mr{N}_{\mr{dg}}(\mr{Ch} (\mca A)^{\circ})$. 
\item If $\mca A$ is the category $\mb{Qcoh}(X) $ of quasi-coherent sheaves on a Noetherian scheme $X$, 
we denote $\ms D(\mb{Qcoh}(X))$ by $\ms D(X)$. 
\item Keep the notation as in $(2)$. 
Define $\ms D_{\mr{coh}}^{b}(X)$ by the full sub-infinity-category of $\ms D(X)$ consisting of bounded complexes with coherent cohomologies. 
\end{enumerate}
\end{dfn}

\begin{prop}[{\cite[Proposition 1.3.5.9]{higheralgebra}}]
Let $\mca A$ be a Grothendieck abelian category. 
Then the infinity category $\ms D(\mca A)$ is stable. 
\end{prop}

\begin{rmk}
The homotopy category $\ho{\ms D(\mca A)}$ is isomorphic to the homotopy category $\mr {Ho}(\K(\mca A))$ with respect to the injective model structure on $\K(\mca A)$.  
Hence the homotopy category $\ho{\ms D(X)}$ is equivalent to the unbounded derived category $\mb D(\mb{Qcoh}\, X)$ of $\mb {Qcoh}(X)$. 

Note that $\ms D^{b}_{\mr{coh}}(X)$ is also stable by \cite[Lemma 1.1.3.3]{higheralgebra} and 
$\ho {\ms D_{\mr{coh}}^{b}(X)}$ is equivalent to the bounded derived category $\mb D^{b}(\mb{coh}\, X)$ of coherent sheaves on a Noetherian scheme $X$. 
\end{rmk}

\begin{rmk}
Basically we use the script font  \verb \mathscr  (ex. $\ms C$) for infinity categories. 
Capital letters with the bold font \verb \mathbf  (ex. $\mb D$) are used for usual categories. 
\end{rmk}

Through this article we use homotopical notation for mapping cones in the triangulated category $\ho {\ms C}$. 
Namely the mapping cone of a morphism $f \colon x \to y $ in $\ho {\ms C}$ is denoted by $\cof f$. 
We also denote $\cof f[-1]$ by $\fib f$. 
Thus we obtain a distinguished triangle in $\ho {\ms C}$ as follows:
\[
\xymatrix{
\fib f	\ar[r]	& x	\ar[r]^f	&	y	\ar[r]	&	\cof f
}. 
\]
We sometimes omit $\cof f$ (or $\fib f$) in the above digram when the triangle is clear from context.

\subsection{The infinity category of morphisms in $\ms C$}
Let $\ms C^{\Delta^{1}}=\Fun{\Delta^{1}}{\ms C}$ be the simplicial set of morphism from the standard simplicial set $\Delta^{1}$ to $\ms C$. 
Then $\ms C^{\Delta^{1}}$ is also an infinity category by \cite[Proposition 1.2.7.3]{MR2522659}. 
Moreover if $\ms C$ is stable then so is $\ms C^{\Delta^{1}}$ by \cite[\S 1.1.1]{higheralgebra}. 
Since the objects in $\ms C^{\Delta^{1}}$ are the morphisms in $\ms C$, 
we refer to $\ms C^{\Delta^{1}}$ as \textit{the infinity category of morphisms} in $\ms C$. 

Since $\ms C$ is isomorphic to $\Fun{\Delta^{0}}{\ms C}$, face maps and degeneracy maps contravariantly induce functors between infinity categories respectively: 
$\xymatrix{
\ms C^{\Delta^{1}}\ar@<1ex>[r]|-{d_{0}}	\ar@<-1ex>[r]|-{d_{1}}&	\ms C
}$
 and 
$\xymatrix{
\ms C \ar[r]	|-{s}&	\ms C^{\Delta^{1}}
}$. 
Here $d_{i}$ is the induced functor from $d^{i}$ ($i=0,1$) and $s$ is the functor induced by $s^{0}$ by abusing notation.

If a morphism $f$ in $ \ms C$ satisfies $d_{0}f=y$ and $d_{1}f=x$, then we write $f$ as $[f \colon x \to y]$ to emphasize the source and the target of the morphism. 
Since there exists a unit transformation $\1_{\ms C^{\Delta^{1}}}\to s \circ d_{0}$, $d_{0}$ is left adjoint to $s$ by \cite[Proposition 5.2.2.8]{MR2522659}. 
Similarly $d_{1}$ is left adjoint to $s$. 
Thus we obtain the following diagram: 
\begin{equation}
\xymatrix{
\ms	C	\ar[r]|(.4)s&\ar@/_8pt/[l]|{d_0}\ar@/^8pt/[l]|{d_1}	\ms C^{\Delta^1}
}
; d_0 \dashv s \dashv d_1. 
\end{equation}

Let $\mr{Mor}(\ho{\ms C})$ be the category of morphisms in $\ho{\ms C}$. 
Then the objects of $\mr{Mor}(\ho{\ms C})$ are the same as that of $\ho{\ms C^{\Delta^{1}}}$ but the morphisms are different. 
To explain the difference, take $f$ and $g$ in $\ms C^{\Delta^{1}}$. 

A morphism $\tau \colon f \to g$ in $\mr{Mor}(\ho{\ms C})$ is a pair $([\tau_{1}], [\tau_{0}])$ of morphisms $[\tau _{i}] \colon d_{i}f \to d_{i}g$ in $\ho{\ms C}$ ($i \in \{0,1\}$) satisfying the equality $[g] \circ [\tau_{1}]= [\tau_{0}]\circ[f]$ in $\Hom_{\ho {\ms C}}(d_{1}f, d_{0}g)$. 
On the other hand a morphism in $\ho{\ms C^{\Delta^{1}}}$ is a pair of higher morphisms in $\ms C$. 
More precisely a morphism $\varphi \colon f \to g$ in ${ \ms C^{\Delta^1}}$ (not in the homotopy category) is 
a pair $(h_{1}, h_{0})$ of $2$-simplices $\Delta^{2}\to \ms C$ such that  
\begin{itemize}
\item $d_{0} h_{1}=g$, $d_{2}h_{0}=f$, and 
\item $d_{1}h_{1}=d_{1}h_{0}(=\psi)$.
\begin{equation}\label{eq:highermorphism}
\xymatrix{
d_1f	\ar[r]^{d_{2}h_{1}}	\ar[d]_{d_{2}h_{0}=f}  \ar[rd]|-{\psi}	&	d_1g \ar[d]^{g=d_{0}h_{1}} \ar@{}[dl]|(.3){h_{1}}	\\
d_0f	\ar[r]_{d_{0}h_{0}} \ar@{}[ur]|(0.3){h_{0}}	&	d_0 g. 
}
\end{equation}
\end{itemize}
Hence a morphisms in the homotopy category is obtained by the the equivalence class $[(h_{1}, h_{0})]$ of the relation in Proposition-Definition \ref{dfn:homotpy}. 

The following lemma is concerned with the difference between $\ho{\ms C^{\Delta^{1}}}$ and $\mr{Mor}(\ho{\ms C})$: 
\begin{lem}\label{lm:difference}

There exists a natural functor $F \colon \ho{\ms C^{\Delta^{1}}} \to \mr{Mor}(\ho{\ms C})$ where $F$ is identity on objects and forgets higher morphisms in $\ms C$. More precisely, 
for a morphism $[\varphi]=[(h_{1}, h_{0})]$ from $f$ to $g$ in $\ho{\ms C^{\Delta^{1}}}$, $F([\varphi])$ is given by the pair $([d_{2}h_{1}], [d_{0}h_{0}])$ of morphisms in $\ho{\ms C}$. 
Morover the functor $F$ is full. 
\end{lem}
 
\begin{proof}
The same argument in the proof of \cite[Propsoition 1.2.3.7]{MR2522659} implies that $F$ is well-defined. 

Let $([\tau_{1}],[\tau_{0} ])$ be a morphism from $f $ to $g$ in $\mr{Mor}(\ho{\ms C})$. 
Choose $1$-simplices $\psi_{1}$ and $\psi_{0}$ such that $[\psi_{1}]=[g]\circ [\tau_{1}]$ and $[\psi_{0}]=[\tau_{0}]\circ [f]$. 
Note that there exist $2$-simplices $h_0 \colon \Delta^{2} \to \ms C$ safisfying $d_{0}h_0 = \tau_{0}, d_{1}h_0=\psi_{0}$, and $d_{2}h_0 =f$ and $\tilde h_{1} \colon \Delta^{2} \to \ms C$ satisfying $d_{0}\tilde h_{1}=g$, $d_{1}\tilde h_{1}=\psi_{1}$ and $d_{2} \tilde h_{1} = \tau_{1}$. 
By the equation $[g]\circ [\tau_{1}]=[\tau_{0}]\circ [f]$, 
we have a $2$-simplex $\rho \colon \Delta^{2} \to \ms C$ such that $d_{0}\rho=\1_{d_{0}g}$, $d_{1}\rho=\psi_{0}$ and $d_{2}\rho=\psi_{1}$. 

Consider the diagram:
\[
\xymatrix{
\Lambda_{2}^{3}\ar[d]	\ar[rr]^{(s_{1}g, \rho, \bullet , \tilde h_{1} )}&	&	\ms C	\\
\Delta^{3}.	\ar@{-->}[urr]_{\sigma}
}
\]
Since $\ms C$ is the infinity category there exists a $3$-simplex $\sigma \colon \Delta^{3} \to \ms C$ as indicated in the above diagram. 
Put $h_{1}=d_{2}\sigma$, then we obtain the desired digram; 
\[
\xymatrix{
d_{1}f	\ar[r]^{\tau_{1}}\ar[d]_{f}	\ar[rd]|{\psi_{0}}	&	d_{1}g\ar[d]^{g} \ar@{}[ld]|(0.3){h_{1}}	\\
d_{0}f \ar[r]_{\tau_{0}}\ar@{}[ur]|(0.32){h_{0}}	&	d_{0}g. 
}
\]
Hence $F$ is full. 
\end{proof}

\begin{rmk}\begin{enumerate}
\item A morphism in $\ho{\ms C^{\Delta^{1}}}$ should be depicted as the diagram(\ref{eq:highermorphism}), 
but we frequently omit $\psi$, $h_{1} $ and $h_{0}$.  
\item Let $ f \to g \to h  $ be a distinguished triangle in $\ho {\ms C^{\Delta^1}}$. 
We obtain the following diagram of distinguished triangles in $\ho{\ms C}$: 
\[
\xymatrix{
d_1f	\ar[r]\ar[d]_f	&	d_1g	\ar[r]	\ar[d]_g	&	d_1h	\ar[d]_h	\\
d_0	f	\ar[r]\ar[d]	&	d_0g	\ar[r]	\ar[d]	&	d_0h	\ar[d]	\\
\cof f	\ar[r]	&	\cof g	\ar[r]	&	\cof h. 
}
\]
In particular the third row is also distinguished triangle in $\ho{\ms C}$. 
\end{enumerate}
\end{rmk}

\subsection{Semiorthogonal decompositions of $\ho{\ms C^{\Delta^{1}}}$}\label{sc:SOD}
The aim of this section is to show that $\ho{\ms C^{\Delta^{1}}}$ has two semiorthogonal decompositions coming from 
two adjoint pairs $d_{0}\dashv s$ and $s \dashv d_{1}$. 
Before the proof, let us recall a semiorthogonal decomposition. 

Let $\mb D$ be a triangulated category. 
A pair $(\mb D_{1}, \mb D_{2})$ of full triangulated subcategories $\mb D_{1}$ and $\mb D_{2}$ of $\mb D$ is said to be a \textit{semiorthogonal decomposition} of $\mb D$ if the pair satisfies 
\begin{itemize}
\item[(1)] $\Hom_{\mb D}(x_{2}, x_{1})=0$ for any $x_{i } \in \mb D_{i}$ ($i=1,2$), and 
\item[(2)] any object $x \in \mb D$ is decomposed into a pair of objects $x_{i}\in \mb D_{i}$ $(i=1,2)$ by the following distinguished triangle in $\mb D$:
\[
\xymatrix{
x_{2}	\ar[r]	&	x	\ar[r]	&	x_{1}	\ar[r]	&	x_{2}[1]. 
}
\]
\end{itemize}
The situation will be denoted by the symbol $\mb D=\< \mb D_{1}, \mb D_{2}\>$ or simply $\< \mb D_{1}, \mb D_{2}\>$. 
In addition to the first condition (1) above, if $\Hom_{\mb D}(x_{1}, x_{2})=0$ holds, the semiorthogonal decomposition is said to be \textit{orthogonal}. 

\begin{rmk}
If $\mb D=\< \mb D_{1}, \mb D_{2}\>$ is a semiorthogonal decomposition of $\mb D$, 
then $\mb D_{1}$ (resp. $\mb D_{2}$) is left admissible (resp. right admissible) by \cite[Lemma 3.1]{MR992977}. 
Namely the inclusion functor $\iota_{1}\colon \mb D_{1}\to \mb D$ (resp. $\iota_{2} \colon \mb D_{2}\to \mb D$) has the left adjoint functor $\tau_{1} \colon \mb D \to \mb D_{1}$ (resp. the right adjoint functor $\tau_{2} \colon \mb D \to \mb D_{2}$). 
Thus we have adjoint pairs of functors $\iota_{2} \dashv \tau_{2}$ and $\tau_{1}\dashv \iota_{1}$. 
\end{rmk}

\begin{lem}\label{lm:SOD}
Let $\ms C$ be a stable infinity  category. 
Set full subcategories of $\ms C^{\Delta^1}$ by 
\begin{align*}
\ms C_{/0}	&=	\left\{ 
[x \to 0]
 \in  {\ms C^{\Delta^1}}	\middle| x \in \ms C \right\},\ 
\ms C_{0/}	=	\left\{ 
[0 \to y]
	\in  {\ms C^{\Delta^1}} \middle| y \in \ms C \right\} \mbox{ and}\\
\ms C_{s}	&=	\left\{ 
[z \stackrel{\1}{\to} z]
 \in { \ms C^{\Delta^1}	}\middle| z \in \ms C \right\}. 
\end{align*}
Furthermore put $\mb C_{/0}:= \ho{\ms C_{/0}}$, $\mb C_{0/}:= \ho{\ms C_{0/}}$ and $\mb C_{s}:= \ho{\ms C_{s}}$
Then the pairs $(\mb D_{0}^{L}, \mb D_{0}^{R}):= (\mb C_{s}, \mb C_{/0})$ and 
$(\mb D_{1}^{L}, \mb D_{1}^{R}):=( \mb C_{0/}, \mb C_{s})$ are semiorthogonal decompositions of $\ho{\ms C^{\Delta^{1}}}$ resepctively. 
\end{lem}

\begin{proof}
Note that an adjoint pair of functors between infinity categories induces the adjoint pair between corresponding homotopy categories by \cite[Proposition 5.2.2.12]{MR2522659}. 

We first prove $\ho{\ms C^{\Delta^{1}}}=\<	\mb C_{s}, \mb C_{/0}\>$. 
Put $[ b \colon y \to 0 ] \in \mb C_{/0}$ and $ [\1_z \colon z \to z] \in \mb C_s $. 
Then we see 
\[
\Hom_{\ho{\ms C^{\Delta^1}}}(b, \1_z)=\Hom_{\ho{\ms C^{\Delta^1}}}(b, s(z) ) \cong 
\Hom_{\ho{\ms C}}(d_0 (b), z) = 
\Hom_{\ho{\ms C}}(0 , z)=0. 
\]
Thus $\Hom_{\ho{\ms C^{\Delta^{1}}}}(b, \1_{z}) =0$ for any $b \in \ms C_{/0}$ and any $\1_{z} \in \mb C_{s}$.

Take an object $[f \colon x \to y]\in \ho {\ms C^{\Delta^1}}$ arbitrary. 
Then the adjunction $d_{0}\dashv s$ implies the canonical morphism $\tau \colon f \to s \circ d_{0}(f)  = \1_{y}$ in $\ho{\ms C^{\Delta^{1}}}$, and we obtain the following distinguished triangle in $\ho{\ms C^{\Delta^{1}}}$:
\begin{equation}\label{fiberseq}
\xymatrix{
\fib f	\ar[r]\ar[d]_{}\ar[rd]	&	x	\ar[r]^{f} \ar[d]_{f}\ar[rd]|{f}		&	y\ar[d]^{\mr{id}}	\\
0		\ar[r]		&	y	\ar[r]_{\mr{id}}			&	y.	\\
}
\end{equation}
Hence we have $\ho{\ms C^{\Delta^{1}}}=\<	\mb C_{s}, \mb C_{/0}\>$

Let $\ms D$ be the opposite category $\ms C^{\mr{op}}$ of $\ms C$. 
Then $\ms D$ is also stable (see also \cite[\S 1.1.1]{higheralgebra}) and any colimit in $\ms D$ is equivalent to the limit in $\ms C$. 
Since $\ms D^{\Delta^{1}}= \mr{Fun}(\Delta^{1}, \ms D)$ is equivalent to the opposite category of $\ms C^{\Delta^{1}}$ by the canonical equivalence $(\Delta^{1})^{\mr{op}} \cong \Delta^{1}$, the functor $d_{0} \colon \ms D^{\Delta^{1}} \to  \ms D$ is the opposite functor of $d_{1} \colon \ms C^{\Delta^{1}} \to \ms C$. 
Hence the argument above for the adjoint pair $d_{0} \dashv s$ also implies $\ho{\ms C^{\Delta^{1}}}=\<\mb C_{0/}, \mb C_{s}\>$. 
\end{proof}

\begin{rmk}\label{convention}
Note that the functor $s \colon \ho{\ms C} \to \mb C_s$ gives an equivalence. 
Since the natural projection $d_{1} \colon \ms C_{/0} \to \ms C$ is a trivial Kan fibration by \cite[\S 1.2.12]{MR2522659}, 
$d_{1}$ has a section denoted by $j_{!} \colon \ms C \to \ms C_{/0}$. 
Similarly a section of $d_{0} \colon \ms C_{0/} \to \ms C$ is denoted by $j_{*} \colon \ms C \to \ms C_{0/}$. 
Taking homotopy categories, the functor $j_{!} \colon \ho{\ms C} \to \mb C_{/0}$ (resp. $j_{*} \colon \ho{\ms C} \to \mb C_{0/}$) gives the inverse functor of the restriction of $d_{1} \colon \mb C_{/0} \to \ho{\ms C}$ (resp. $d_{0} \colon \mb C_{0/} \to \ho{\ms C}$). 
Throughout this article we always identify $\mb C_{/0}, \mb C_{0/}$ and $\mb C_s$ with $\ho{\ms C}$ via $j_{!}$, $j_{*}$ and $s$ respectively. 
In particular we have the following explicit descriptions of $j_{*}$ and $j_{!}$: 
\begin{equation}\label{eq:inverse}
j_{!} \colon \ho{\ms C} \stackrel{\sim}{\to} \mb C_{/0}; j_{!}(x)=[x\to 0],\mbox{ and  } 
j_{*} \colon \ho{\ms C} \stackrel{\sim}{\to} \mb C_{0/}; j_{*}(x)=[0\to x]. 
\end{equation}
\end{rmk}

\subsection{Serre functors on $\ho {\ms C^{\Delta^{1}}}$}
We observe the Serre functor on $\ho{\ms C^{\Delta^{1}}} $ under certain assumption on $\ms C$. 
As a consequence if $\ms C$ is the infinity category $\ms D_{\mr{coh}}^{b}(X)$ of a smooth projective variety $X$, we have an explicit description of the Serre functor on $\ho {\ms C^{\Delta^{1}}}$. 
Before the argument, recall that $\mb k$-linear category $\mb D$ is said to be \textit{finite} if any vector space $\Hom_{\mb D}(x,y)$ is finite dimensional for all $x,y \in \mb D$.

\begin{prop}\label{pr:k-linear}
Let $\ms D$ be a dg category over a field $\mb k$. 
\begin{enumerate}
\item Then the homotopy category $\ho{\mr{N}_{\mr {dg}}(\ms D)^{\Delta^{1}}}$ is a $\mb k$-linear category. 
\item If $\mr{N}_{\mr {dg}}(\ms D)$ is stable and $\ho {\mr{N}_{\mr {dg}}(\ms D)}$ is finite, then $\ho {\mr{N}_{\mr {dg}}(\ms D)^{\Delta^{1}}}$ is also finite. 
\end{enumerate}
\end{prop}

\begin{proof}
Put $\ms C= \mr{N}_{\mr{dg}}(\ms D)$ and let $f$ and $g$ be in $\ms C^{\Delta^{1}}$. 
Then $f$ is in $Z^{0}_{\ms D}(d_{1}f, d_{0}f)$ and so does $g$ by Remark \ref{rmk:mugen-dg}. 
Note that an $n$-simplex of $\ms C^{\Delta^{1}}$ is a morphism $\Delta^{n} \times \Delta^{1} \to \ms C$ as simplicial sets. 
According to Proposition-Definition \ref{dfn:dg-nerve}, a morphism from $f $ to $g$ in $\ms C^{\Delta^{1}}$ is a 5-tuple $(\tau_{1}, \tau_{0}, \psi, h_{1}, h_{0})$ where 
\begin{itemize}
\item $\tau_{i} \in  Z^{0}_{\ms D}(d_{i}f, d_{i}g)$, $\psi \in Z^{0}_{\ms D}(d_{1}f, d_{0}g) $, 
\item  $h_{1} \in \Map_{\ms D}^{-1}(d_{1}f, d_{0}g)$  with $\delta( h_{1})=g\tau_{1}-\psi$, and 
\item  $h_{0} \in \Map_{\ms D}^{-1}(d_{1}f, d_{0}g)$  with $\delta( h_{0})=\tau_{0}f-\psi $. 
\end{itemize}
Let $ \mr{Hom}(f,g)$ be the set of 5-tuples $(\tau_{1}, \tau_{0}, \psi, h_{1}, h_{0}) $ satisfying the above relation. 
Then $\mr{Hom}(f,g)$ is naturally a $\mb k$-linear space via termwise addition and $\mb k$-action.

By Proposition-Definition \ref{dfn:homotpy}, we see that two 5-tuples $	(\tau_{1}, \tau_{0}, \psi, h_{1}, h_{0}) $ and $	(\tilde{\tau}_{1}, \tilde{\tau}_{0}, \tilde{\psi}, \tilde{h}_{1}, \tilde{h}_{0}) $ are equivalent if there exists a $7$-tuple 
$h_{\overline{012}} \in \Map_{\ms D}^{-1}(d_{1}f, d_{1}g)$, $\{ h_{\overline{01} \underline 2}, h_{\overline 0 \underline {12}}   	\} \subset \Map_{\ms D}^{-1}(d_{1}f, d_{0}g)$, $h_{\underline{012}} \in \Map_{\ms D}^{-1}(d_{0}f, d_{0}g)$ 
and $\{ h_{\overline {012} \underline 2}, h_{\overline{01} \underline 1 \underline 2}, h_{\overline{0} \underline{012} } \} \subset  \Map_{\ms D}^{-2}(d_{1}f, d_{0}g)$ satisfying
\begin{itemize}
\item $\delta (h_{\overline{012}})=  \tau_{1} - \tilde \tau_{1}$, 
\item $\delta (h_{\underline{012}})=  \tau_{0}-\tilde \tau_{0}$, 
\item $\delta (h_{\overline{01}\underline{2}})= g  \tau _{1}- \tilde \psi$, 
\item $\delta (h_{\overline{0}\underline{12}})=  \psi - \tilde \psi  $, 
\item $\delta (h_{\overline{012}\underline{2}})=-h_{\overline{01}\underline{2}}+ g  h_{\overline{012}} + \tilde {h}_{1}  $,
\item $\delta (h_{\overline{01}\underline{12}})= -h_{\overline{01}\underline{2}} + h_{1} +  h_{\overline{0}\underline{12}} $, and 
\item $\delta (h_{\overline0 \underline{012}})= - \tilde h_{0} +  h_{0} + h_{\overline{0}\underline{12}} - h_{\underline{012}}  f$. 
\end{itemize}
Since the set of collections $(\tau_{1}, \tau_{0}, \psi, h_{1}, h_{0})$ which are equivalent to zero of $\mr{Hom}(f,g)$ is a $\mb k$-linear subspace of $\mr{Hom}(f,g)$, $\ho {\ms C^{\Delta^{1}}}$ is a $\mb k$-linear category. 

To prove the second assertion, 
consider the semiorthogonal decomposition of $\ho {\ms C^{\Delta^{1}}} = \<\mb D_{0}^{L}, \mb D_{0}^{R}\>$ in Lemma \ref{lm:SOD}. 
Then we have the distinguished triangle 
\[
\xymatrix{
j_{!} \circ \fib (f)	\ar[r]	&	f	\ar[r]	&	s \circ d_{0}(f)	\ar[r]	&	j_{!} \circ \fib (f)[1]
}
\]
for any $f \in \ms C^{\Delta^{1}}$, and the exact sequence of $\mb k$-vector spaces
\begin{equation}\label{eq:finite}
\xymatrix{
\Hom_{\ho{\ms C^{\Delta^{1}}}}(s \circ d_{0}f, g)	\ar[r]	&	\Hom_{\ho{\ms C^{\Delta^{1}}}}(f, g) \ar[r]	&	\Hom_{\ho{\ms C^{\Delta^{1}}}}\left(j_{!} \circ \fib (f), g\right) 
}
\end{equation}
for any $g $ in $\ms C^{\Delta^{1}}$. 
The right term in (\ref{eq:finite}) is isomorphic to $\Hom_{\ho{\ms C}}(\fib f, \fib g)$. 
Moreover the left term in (\ref{eq:finite}) is isomorphic to $\Hom_{\ho{\ms C}}(d_{0}f, d_{1}g)$ since the functor $d_{1} $ is the right adjoint of $s$. 
Then the assumption implies that 
the category $\ho{\mr{N}_{\mr {dg}}(\ms D)^{\Delta^{1}}}$ is finite.  
\end{proof}

\begin{prop}\label{pr:BK}
Let $\ms D$ be a $\mb k$-linear differential graded category and put $\ms C= \mr N_{\mr{dg}}(\ms D)$. 
Suppose that $\ms C$ is stable and $\ho {\ms C}$ is finite. 
If $\ho {\ms C}$ has a Serre functor $S_{\ms C}$ then $\ho {\ms C^{\Delta^{1}}}$ has also a Serre functor $S_{\ms C^{\Delta^{1}}}$ given by 
\begin{equation}
S_{\ms C^{\Delta^{1}}}(f) = [S_{\ms C}(u) \colon S_{\ms C}(d_{0}f) \to S_{\ms C}(\cof f)],  \label{Serre}
\end{equation}
where $u$ is the universal morphism $y \to \cof f$ in $\ho{\ms C}$. 
\end{prop}

\begin{proof}
According to Bondal-Kapranov \cite[Proposition 3.8, Theorem 2.10]{MR1039961}, 
let $\mca B$ be the essential image $\mb C_{s}$ of $s\colon \ho {\ms C} \to \ho {\ms C^{\Delta^{1}}}$. 
Since the right (resp. left) adjoint of $s$ is $d_{1}$ (resp. $d_{0}$), $\mca B$ is right and left admissible in $\ho {\ms C^{\Delta^{1}}}$. 
Moreover the right adjoint of the inclusion $j_{*} \colon \ho {\ms C }\to \mca B^{\perp} \to \ho{\ms C^{\Delta^{1}}}$ is given by  $d_{0} \colon \ho{ \ms C^{\Delta^{1}} }\to \ho {\ms C}$. 
Thus $\mca B^{\perp}$ is also right and left admissible. 
Since $\mca B$ and $\mca B^{\perp}$ are both equivalent to $\ho {\ms C}$, the triangulated category $\ho {\ms C^{\Delta^{1}}}$ has the Serre functor by \cite[Proposition 3.8]{MR1039961}.

It is enough to show that the functor $ h(-)= \Hom_{\ho  {\ms C^{\Delta^{1}}}}(f, -)^{\vee} \colon \ho {\ms C^{\Delta^{1}}} \to \mr {Vect}_{\mb k}$ is represented by $S_{\ms C}(u)$ in (\ref{Serre}). 
Since $\mca B \cong \ho {\ms C}$ has the Serre functor, the contravariant functor  
$\Hom_{\ho {\ms C}}(x, -)^{\vee} \colon \ho {\ms C} \to \mr {Vect}_{\mb k}$ 
is represented by $S_{\ms C}(x)$ where $\mr {Vect}_{\mb k}$ is the category of $\mb k$-vector spaces.

Take the semiorthogonal decomposition $\ho {\ms C ^{\Delta^{1}}} = \< \mca B, {}^{\perp}\mca B   \>=\<\mb C_{s}, \mb C_{/0} \>$. 
Then we obtain the following distinguished triangle in $\ho{\ms C^{\Delta^{1}}}$:
\[
\xymatrix{
\fib f	\ar[r]\ar[d]	&	x	\ar[r]\ar[d]_{f}	&	y	\ar[r]\ar[d]_{\1}	&	\fib f[1]\ar[d]	\\
0	\ar[r]			&	y		\ar[r]	&	y	\ar[r]	&	0
}
\]
Then the restriction $h|_{\mca B}$ 
is represented by the object $E= [\1 \colon S_{\ms C}(y) \to S_{\ms C}(y)  ]$ in $\ho{\ms C^{\Delta^{1}}}$. 

Take the semiorthogonal decomposition $\< \mca B^{\perp}, \mca B \> = \<\mb C_{0/} , \mb C_{s}	\>$. 
Then we obtain the following:
\[
\xymatrix{
x	\ar[r]\ar[d]_{\1}		&	x	\ar[r]\ar[d]_{f}		&	0	\ar[r]\ar[d]	&	x[1] \ar[d]_{\1[1]}	\\		
x	\ar[r]	&	y\ar[r]	&	\cof f	 \ar[r]	&	x[1]
}
\]
The same argument above implies that the restriction $h|_{\mca B^{\perp}}$ of $h$ to $\mca B^{\perp}$ is represented by the object $F= j_{*}\left(S_{\ms C}(\cof f)\right) = [ 0 \to S_{\ms C}(\cof f)]$ in $\ho{\ms C^{\Delta^{1}}}$.

The adjunction $d_{0} \dashv s$ implies that the restriction functor of the representable functor by the morphism $E$ to $\mca B^{\perp}$ is represented by $ F' = [0 \to S_{\ms C}(y)]$. 
Namely we have 
\begin{equation}
\label{BK} \Hom_{\ho {\ms C^{\Delta^{1}}}}(- , E)|_{\mca B^{\perp}} \cong \Hom_{\mca B^{\perp}}(-, F'). 
\end{equation}
Then evaluation of (\ref{BK}) at $F'$ implies the canonical morphism $\gamma \colon F' \to E$ which satisfies $d_{0}(\gamma)=\1$. 

Let $\varphi \colon F' \to F$ be the morphism defined by $h(\gamma )(\1_{E})  \in h(F') \cong \Hom_{\ho {\ms C^{\Delta^{1}}}}(F',F)$. 
Since $h(E)$ (resp. $h(F')$) is canonically isomorphic to $\Hom_{\ho {\ms C}}(y,y)$ (resp. $\Hom_{\ho {\ms C}}(y,\cof f)$), we see $d_{0}\varphi = S_{\ms C}(u)$ where $u$ is  a universal morphism $u \colon y \to \cof f$ in $\ms C$.  
Then we obtain the following distinguished triangle in $\ho {\ms C^{\Delta^{1}}}$:
\[
\xymatrix{
0	\ar[r]^{d_{1}\varphi}\ar[d]_{F'}	&	0	\ar[r]\ar[d]_{F}	&	0	\ar[r]\ar[d]	&	0\ar[d]	_{F'[1]}\\
S_{\ms C}(y)	\ar[r]^{d_{0}\varphi}	&	S_{\ms C}(\cof f)\ar[r]	& S_{\ms C}(x)[1] \ar[r]	&	S_{\ms C}(y)[1]
}
\]
Put $F''=  [0 \to S_{\ms C}(x)[1]]$. 
Let $\delta \colon F'' \to E[1]$ be the composite $F'' \to F'[1] \to E[1]$ and $X$ the mapping cone of $\delta$. 
Then $X$ is isomorphic to $S_{\ms C}(u)$ in (\ref{Serre}). 
Hence the functor $h(-) =\Hom_{\ho  {\ms C^{\Delta^{1}}}}(f, -)^{\vee}$ is represented by $S_{\ms C}(u)$ via the proof of \cite[Theorem 2.10]{MR1039961}. 
\end{proof}

\begin{rmk}
If the Serre functor $S_{\mb D}$ on a $\mb k$-linear triangulated category $\mb D$ is isomorphic to the $d$-times shifts $[d]$, then $\mb D$ is said to be $d$-dimensional Calabi-Yau category. 
Moreover if $k$-times composite $S_{\mb D}^{k}$ is isomorphic to the shifts $[d]$, then $\mb D$ is said to be $d/k$-dimensional Calabi-Yau or a fractional Calabi-Yau. 
One of the simplest example of fractional Calabi-Yau categories is 
the bounded derived category of the finite dimensional representations of $A_{2}$-quiver. 
Such a category is obtained by $\ho{\ms D_{\mr{coh}}^{b}(\Spec \mb k)}$ of a field $\mb k$ (See also Remark \ref{rmk:relevant}). 

From Proposition \ref{pr:BK}, we see 
\[
S_{\ms C^{\Delta^{1}}}^{3}([f \colon x \to y]) = \left[S_{\ms C}^{3} (f[1]) \colon S^{3}_{\ms C} (x[1]) \to S^{3}_{\ms C}(y[1])\right]. 
\]
Thus $\ho {\ms C^{\Delta^{1}}}$ is $(3d+1)/3$-dimensional Calabi-Yau category if $\ho {\ms C}$ is $d$-dimensional Calabi-Yau. 
\end{rmk}

\section{An observation on  Problem \ref{problem1}}

\subsection{Notation for stability conditions}
Let $\mb D$ be a triangulated category such that the rank of Grothendieck group $K_{0}(\mb D)$ of $\mb D$ is finite. 
We follow notation and a basic definition of stability conditions on a triangulated category $\mb D$ from the original article \cite{MR2373143} due to Bridgeland. 
For instance $\Stab {\mb D}$ is the set of locally finite stability conditions on $\mb D$. 
The central charge of $\sigma \in \Stab {\mb D}$ is denoted by $Z$. 
$\mca A$ and $\mca P$ are respectively the heart and the slicing of $\sigma$. 
The set $\Stab {\mb D}$ has a topology and each non-empty connected component is a complex manifold whose dimension is smaller than or equal to the rank of $K_0(\mb D)$. 
We also recall that $\Stab {\mb D}$ has a left action of $\Aut {\mb D}$ and a right action of the universal cover $\widetilde{\mr {GL}}^+_2(\bb R)$ of $\mr{GL}^+_2(\bb R)$. 
The right action of $\widetilde{\mr {GL}}^+_2(\bb R)$ is continuous.

A stability condition $\sigma$ on $\mb D$ is said to be \textit{full} if the tangent space of $\Stab{ \mb D}$ at $\sigma$ has the maximal dimension $\rank K_0(\mb D)$. 
A connected component of full stability conditions on $\mb D$ is said to be a \textit{full component}.  

One of basic properties of stability conditions is the \textit{support property}: 

\begin{dfn}[Support property]\label{supportproperty}
Let $\mb D$ be a triangulated category with $\rank K_{0}(\mb D) < \infty$, 
and let $\|	- \| \colon K_{0}(\mb D) \otimes \bb R \to \bb R$ be a norm. 
Then a stability condition $\sigma$ on $\mb D$ satisfies the \textit{support property} if the following holds:
\[
\exists C > 0 \mbox{ such that } \sup
\left\{  \frac{\|E \|}{|Z(E)|}  \middle|  E\mbox{ is $\sigma$-semistable}   \right\} \leq C. 
\]
\end{dfn}
Since any norm on $K_0(\mb D)\otimes \bb R$ is equivalent, the support property is independent of the choice of the norm. 
The following lemma gives us a transparent understanding of fullness and support property.

\begin{lem}[{\cite[Appendix B.4]{MR2852118}}]\label{BMb.4}
Let $\mb D$ be a triangulated category. 
Assume $\rank K_0 (\mb D)$ is finite. 
Then a locally finite stability condition $\sigma$ is full if and only if $\sigma$ has the support property. 
\end{lem}

\subsection{A fundamental observation}

We give an observation of Problem \ref{problem1} in Proposition \ref{pr:easiest}. 
The following is a key ingredient of the observation. 

\begin{lem}\label{lm:indecomposable}
Let $\mb D$ be an additive category. 
If any non-zero endomorphism of an object $E \in \mb D$ is invertible, then $E$ is indecomposable. 
\end{lem}

\begin{rmk}
In the earlier version of Lemma \ref{lm:indecomposable}, the category $\mb D$ was assumed to be a bounded derived category. 
The generalization to additive categories was given by the referee. 
\end{rmk}

\begin{proof}
Suppose that the object $E$ is $A \+ B$ with $A \neq 0$. 
Then it is enough to show $B=0$. 

Let $p \colon A\+ B \to A$ the projection and $i \colon A \to A\+ B$ the section of $p$. 
Then the composite $e= i \circ p$ is non-zero and invertible since $p\circ i \circ p=p$. 
Moreover we have $e^{2}=e$ which implies $e=\1_{E}$ since $e$ is invertible. 
On the other hand we have $e(B)=0$. 
Since $e=\1_{E}$, $B$ has to be $0$. 
%
%
%
\end{proof}

Lemma \ref{lm:indecomposable} supplies an example of a derived category which has no stability condition.

\begin{prop}\label{pr:empty}
Suppose that a principle integral domain $R$ is not a field. 
Let $\mb D^{b}(\mb{mod}\,  R)$ be the bounded derived category of the abelian category $\mb{mod}\, R$ of finitely generated $R$-modules. 
Then $\Stab{\mb D^{b}(\mb{mod}\, R)}$ is empty. 
\end{prop}

\begin{proof}
Suppose to the contrary that there exists a locally finite stability condition $\sigma$ on $\mb D^{b}(\mb{mod}\,  R)$. 
We can take a $\sigma$-stable object $E \in \mb D^{b}(\mb{mod}\,  R)$ since $\sigma $ is locally finite. 
Then any non zero endomorphism of $E$ is invertible and Lemma \ref{lm:indecomposable} implies $E$ is indecomposable. 
Since the global dimension of $R$ is $1$, any indecomposable object in $\mb D^{b}(\mb{mod}\,  R)$ is a shifts of an indecomposable object in $\mb{mod}\, R$. 
Moreover any indecomposable object in $\mb{mod }\, R$ is $R$ or $R/(r)$ for some non zero element $r \in R$ by the assumption for $R$. 

Without loss of generality, we assume $E=R$ or $E=R/(r)$ where $r \in R\setminus\{0\}$. 
Recall that $R$ itself has a non-invertible endomorphism as $R$-modules by the assumption. 
Thus $R$ never be $\sigma$-stable since any non-zero endomorphism of a $\sigma$-stable object is invertible. 

Then any $\sigma$-stable object is isomorphic to a torsion $R$-module up to shifts. 
The existence of the Harder-Narasimhan filtrations implies that the module $R$ itself is given by a successive extension of shifts of torsion $R$-modules. 
This clearly gives a contradiction since the rank of $R$ is positive but the rank of torsion modules is zero. 
\end{proof}

\begin{cor}
Let $\{R_{i}\}_{i=1}^{n}$ be the finite set of principle ideal domains which are not fields and put $R= \prod_{i=1}^{n}R_{i}$. 
Then $\Stab{\mb D^{b}(\mb{mod}\, R)}$ is empty. 
\end{cor}

\begin{proof}
From the assumption, the derived category $\mb D^{b}(\mb{mod}\, R)$ has the orthogonal decomposition: 
 $\mb D^{b}(\mb{mod}\, R) =\bigoplus_{i=1}^{n} \mb D^{b}(\mb{mod}\, R_{i})$. 
 By \cite[Proposition 5.2]{MR3984103}\footnote{Thought the article \cite{MR3984103} has been written with assuming a triangulated category is linear over algebraically closed field, one can prove \cite[Proposition 5.2]{MR3984103} without the assumption. Hence we can apply the proposition. }, we see $\Stab{\mb D^{b}(\mb{mod}\, R)}$ is isomorphic to the product $\prod_{i=1}^{n}\Stab{\mb D^{b}(\mb{mod}\, R_{i})}$. 
Thus the category 
$\mb D^{b}(\mb{mod}\, R)$ has no locally finite stability condition. 
\end{proof}

Now we go back to an observation for Problem \ref{problem1}.

\begin{prop}\label{pr:easiest}
Let $\ms C$ be the infinity category $\ms D_{\mr{coh}}^{b}(\Spec \mb k)$ 
of the scheme $\Spec \mb k$ where $\mb k$ is a field (cf. Definition \ref{dfn:derivedmugen}). 
Then both $\Stab  {\ho {\ms C}}$ and $\Stab {\ho {\ms C^{\Delta^1}}}$ are contractible. 
In particular the answer for Problem \ref{problem1} is affirmative. 
\end{prop}

\begin{proof}
Since any nonzero complex number gives an orientation preserving $\bb R$-liner isomorphism on $\bb R^{2}$, the multiplicative group $\bb C^{*}=\bb C\setminus \{0\}$ is a subgroup of $\mr{GL}_{2}^{+}(\bb R)$. 
The universal cover $\bb C$ of $\bb C^{*}$ is a subgroup of $\widetilde{\mr{GL}}_{2}^{+}(\bb R)$. 
Thus $\bb C$ acts on $\Stab{\mb  D}$ for any triangulated category $\mb D$. 
The right action of $\bb C$ on $\Stab{\mb D}$ is explicitly given by
\begin{equation}\label{C-action}
\sigma \cdot z = ( W, \mca Q), 
W(E)= \exp(-z) Z(E) \mbox{ and }\mca Q(\phi) = \mca P(\phi +y /\pi), 
\end{equation}
where $\sigma=(Z, \mca P)$ and $z=x+\sqrt{-1}y$. 
In particular the $\bb C$ action on $\Stab{\mb D}$ is holomorphic. 

We show that $\Stab{\ho {\ms C}}$ is isomorphic to $\bb C$ as complex manifolds. 
It is enough to show that the $\bb C$ action on $\Stab{\ho {\ms C}}$ is free and transitive. 
Let $E$ be a stable object for a stability condition $\sigma \in \Stab {\ho {\ms C}}$. 
Since $E$ is indecomposable by Lemma \ref{lm:indecomposable}, 
$E$ should be isomorphic to $\mb k$ up to shifts. 
Hence $\mb k$ is stable for all $\sigma \in \Stab {\ho{ \ms C}}$ and the $\bb C$ action on $\Stab{\ho {\ms C}}$ is free by the $\bb C$ action (\ref{C-action}). 

Take stability conditions $\sigma_{1}$ and $\sigma _{2}$. 
Let $\phi_{i}$ is the phase of $\mb k$ in $\sigma_{i}$ and $m_{i}$ be the mass of $\mb k$ in $\sigma_{i}$ ($i \in \{1,2\}$). 
Put $z = \log m_{1}-\log m_{2} + \sqrt{-1}\pi(\phi_{1}-\phi_{2})$. 
Then we see $\sigma_{1} \cdot z= \sigma _{2}$, since $\mb k$ is stable for any $\sigma \in \Stab{\ho{\ms C}}$. 
Thus the $\bb C$-action is transitive.

By Corollary \ref{pr:MMR}, $\ho{\ms C^{\Delta^1}}$ is the bounded derived category $\mb D^{b}\left(\mb{mod}(\bullet  \to 	\bullet)\right)$ of the finite dimensional representations of the $A_2$ quiver $\bullet \to \bullet$.  
The argument due to Macr\`{i} \cite{MR2335991} essentially implies that $\Stab {\ho {\ms C^{\Delta^1}}}$ is contractible as follows. 

According to \cite{MR2335991}, $\Stab {\ho {\ms C^{\Delta^1}}}$ has an open covering $\mca U = \{ \Theta_i  \}_{i=0}^2$ such that 
each $\Theta _i$ is contractible and any two intersections are the same $\widetilde{\mr{GL}}_2^+ (\bb R)$-orbit $O_{-1}$ of a stability condition:$\Theta_i \cap \Theta _j = O_{-1}$. 
Then any finite intersection is contractible since $O_{-1}$ is homeomorphic to $\widetilde{\mr{GL}}_2^+ (\bb R)$. 
The nerve theorem \cite{MR0050280} (or \cite[Corollary 4G.3]{MR1867354}) implies that $\Stab {\ms C^{\Delta^1}}$ is homotopy equivalent to $N(\mca U)$ where $N(\mca U)$ is the nerve of the covering $\mca U$. 
Clearly $N(\mca U)$ is the standard simplicial complex $\Delta^2$ (here $\Delta^2$ is not a simplicial set but a simplicial complex). 
Hence $\Stab {\ho {\ms C^{\Delta^1}}}$ is contractible. 
\end{proof}

Through the proof above, we have proved the following: 

\begin{cor}\label{cor:speck}
Let $\mb D^{b}(\mb{mod}\, \mb k)$ be the bounded derived category of finite dimensional $\mb k$-vector spaces where $\mb k$ is a field. 
Then $\Stab{\mb D^{b}(\mb{mod}\, \mb k)}$ is isomorphic to $\bb C$. 
\end{cor}

\begin{rmk}
It is difficult to generalize Proposition \ref{pr:easiest} to arbitrary case since we only calculate the homotopy types of 
$\Stab {\ho {\ms C}}$ and $\Stab {\ho {\ms C^{\Delta^1}}}$ independently. 
Thus our basic motivation of this note is a construction of continuous maps between $\Stab {\ho {\ms C}}$ and $\Stab {\ho {\ms C^{\Delta^1}}}$ for (arbitrary) stable infinity categories. 
\end{rmk}

\section{Stability conditions on morphisms}\label{sec:4}

The aim of this section is the construction of continuous maps form $\Stab {\ho {\ms C}}$ to $\Stab {\ho {\ms C^{\Delta^1}}}$ and we wish to study some properties of the morphisms. 
A key ingredient of  the construction is the ``gluing construction" developed by \cite{MR2721656}.

\subsection{Gluing construction of stability conditions}

As observed in \S \ref{sc:SOD}, $\ho {\ms C^{\Delta^1}}$ has semiorthogonal decompositions. 
In particular each component is equivalent to $\ho {\ms C}$ itself.  
Collins--Polishchuck \cite{MR2721656} proposed a construction of stability conditions on a triangulated category $\mb D$ coming from a semiorthogonal decomposition $\mb D= \<\mb D_1, \mb D_2  \>$. 
A key ingredient of the construction is a \textit{reasonable} stability condition on $\ho {\ms C}$:

\begin{dfn}[{\cite[pp. 568]{MR2721656}}]\label{dfn-reasonable}
A stability condition $\sigma =(\mca A, Z)$ on a triangulated category $\mb D$ is 
\textit{reasonable} if $\sigma$ satisfies
\[
0 < \inf	\{	|Z(E)| \in \bb R  \mid 	E\mbox{ is semistable in }\sigma  \}. 
\]
We also denote by $\Stabr {\mb D}$ the set of reasonable stability conditions on $\mb D$. 
\end{dfn}

\begin{rmk}
Since a reasonable stability condition is locally finite by \cite[Lemma 1.1]{MR2721656}, $\Stabr{\mb D}$ is a subset of $\Stab{\mb D}$. 
Unfortunately we do not know whether $\Stabr {\mb D} = \Stab {\mb D}$ or not. 

If a stability condition $\sigma =(Z,\mca P)$ satisfies the support property then $|Z(E)|$ has lower bound since a norm on $K_0(\mb D)$ is discrete. 
Hence $\Stabr {\mb D}$ contains any full components of $\Stab {\mb D}$ by Lemma \ref{BMb.4}. 
Thus a reasonable stability condition is sufficiently ``reasonable". 
\end{rmk}

\begin{prop}[{\cite{MR2721656}}]\label{CP2.2}
Suppose that the triangulated category $\mb D$ has a semiorthogonal decomposition $\<\mb D_1, \mb D_2  \>$. 
The left adjoint of the inclusion $\mb D_1 \to \mb D$ is denoted by $\tau_1$ and the right adjoint of the inclusion $\mb D_2 \to \mb D$ is denoted by $\tau_2$.
Let $\sigma_i = (Z_i, \mca P_i)$ be a reasonable stability condition on $\mb D_i$ for $i \in \{1,2\}$.  
Suppose that $\sigma_1$ and $\sigma_2$ satisfy the following conditions
 \begin{enumerate}
\item  $\Hom_{\mb D}^{\leq 0}\left (\mca P_1(0,1], \mca P _2(0,1]\right) =0$ \label{condition-a} and 
\item There is a real number $a\in (0,1)$ such that $\Hom_{\mb D}^{\leq 0} \left(\mca P_1(a,a+1], \mca P _2(a,a+1]\right) =0$. \label{condition-b}
\end{enumerate}

Then there exists a unique reasonable stability condition $\gl{\sigma_1} {\sigma_2}$ on $\mb D$ glued from $\sigma _1$ and $\sigma _2$ whose heart $\mca A$ of the $t$-structure of $\gl{\sigma _1} {\sigma _2}$ is given by 
\[
\mca A= \{ E \in \mb D \mid \tau_i(E) \in \mca P_i\left((0,1]\right)  \, (i=1,2)\}
\]
and whose central charge $Z$ is given by $Z(E) = Z_1(\tau_1(E)) + Z_2(\tau_2(E))$. 
Moreover the gluing construction is continuous on pairs $\pair {\sigma_1}{\sigma_2}$ satisfying the above conditions (\ref{condition-a}) and (\ref{condition-b}). 
\end{prop}

\begin{lem}\label{lm:CP}
Notation is the same as in Proposition \ref{CP2.2}. 
Suppose that $\gl{\sigma _1}{ \sigma _2} =(Z, \mca P) $ is the glueing stability condition. Then we have 
\[
\mca P(\phi ) \cap \mb D_i = \mca P_i(\phi)
\]
\end{lem}

\begin{proof}
According to \cite[Proposition 2,2]{MR2721656}, we have $\mca P_i(\phi) \subset \mca P(\phi)$. 
Thus we see $\mca P_i(\phi) \subset \mca P(\phi) \cap \mb D_i$. 

Conversely let $E \in \mca P(\phi) \cap \mb D_i$. 
Suppose to the contrary that $E$ is not $\sigma_i$-semistable. 
Take the first term $A$ of the Harder-Narasimhan filtration of $E$ with respect to $\sigma_{i}$. 
Then $A$ is $\sigma_{i}$-semistable with $\arg Z_i(A) > \phi $. 
Since $A$ is also $\gl {\sigma _1} {\sigma _2}$-semistable in phase $\arg Z_i(A)$ by \cite[Proposition 2.2]{MR2721656}, there is no morphism from $A$ to $E$. 
Hence $E$ should be $\sigma_i$-semistable and this give the proof of the opposite inclusion $\mca P_i (\phi ) \supset \mca P(\phi) \cap \mb D_i$. 
\end{proof}

\begin{prop}\label{pr:pullback}
Let $\ms C$ be a stable infinity category with $\rank K_{0}(\ho{\ms C}) < \infty$. 
Then there exists continuous maps as follows: 
\[
d_0^* , d_1^*  \colon \xymatrix{	\Stabr {\ho {\ms C}} \ar@<0.5ex>[r]	\ar@<-0.5ex>[r]	&	\Stabr {\ho {\ms C^{\Delta^1}}}}
\]
\begin{enumerate}
\item 
Let $\ho{\ms C^{\Delta^{1}}}=\< \mb D_{0}^{L}, \mb D_{0}^{R}\>$ be the semiorthogonal decomposition in Lemma \ref{lm:SOD}.  
We define the pull back of $\sigma $ along the functor $d_0$ by
\[
d_0^* \sigma := \gl{\sigma _L} {\sigma _R}= \gl{\sigma} {[-1]\sigma}. 
\]
\item Let $\ho{\ms C^{\Delta^{1}}}=\< \mb D_{1}^{L}, \mb D_{1}^{R}\>$ be the semiorthogonal decomposition in Lemma \ref{lm:SOD}. 
We define the pull back of $\sigma $ along the functor $d_1$ by
\[
d_1^* \sigma := \gl{\sigma _L} {\sigma _R} = \gl {[1]\sigma} {\sigma } . 
\]
\end{enumerate}
\end{prop}

\begin{proof}
Let $\mca P$ be the slicing of $\sigma \in \Stabr{\ho{\ms C}}. $
Then the heart of $\sigma$ and $[-1]\sigma$ are respectively $\mca P\left((0,1]\right)$ and $\mca P\left((-1,0]\right)$. 
Thus the conditions (\ref{condition-a}) and (\ref{condition-b}) in Proposition \ref{CP2.2} holds since 
\[
\Hom_{\ho {\ms C^{\Delta^1}}}\left(s(x), j_!(y)\right)\cong \Hom_{\ho{\ms C}}(x, d_{1}\circ j_{!}(y))=\Hom_{\ho {\ms C}}(x,y)
\]
 for $x, y \in \ho{\ms C}$.
 Hence $d_0^* \sigma =\gl {\sigma} {[-1]\sigma}$ does exists. 
Continuity of the gluing construction follows from \cite[Theorem 4.3]{MR2721656}. 
The second assertion is similar. 
\end{proof}

\begin{lem}\label{lm:4.6}
Let $\sigma=(\mca A, Z)$ be in $\Stabr{\ho{\ms C}}$ and let $d_{i}^{*}\sigma=(d_{i}^{*}\mca A, d_{i}^{*}Z)$ be the pull back of $\sigma$ by $d_{i}$ for $i \in \{1,2 \}$. 
\begin{enumerate}
\item An object $f \in \ho{\ms C^{\Delta^{1}}}$ is in $d_{0}^{*}\mca A$ if and only if both $d_{0}f$ and $\cof f$ are in $\mca A$. 
\item If an object $x$ is in $\mca A$, then $\arg d_{0}^{*}Z(s(x))=\arg Z(x)$ and $\arg d_{0}^{*}Z(j_{!}(x[-1]))=\arg Z(x)$. 
\item An object $f \in \ho{\ms C^{\Delta^{1}}}$ is in $d_{1}^{*}\mca A$ if and only if both $d_{1}f$ and $\fib f$ are in $\mca A$. 
\item If an object $x$ is in $\mca A$, then $\arg d_{1}^{*}Z(s(x))=\arg Z(x)$ and $\arg d_{1}^{*}Z(j_{*}(x[-1]))=\arg Z(x)$. 
\end{enumerate}
\end{lem}

\begin{proof}
The construction of $d_{0}^{*}\sigma$ and $d_{1}^{*}\sigma $ directly implies the assertions. 
\end{proof}

\begin{prop}\label{pr:fundamental}
Let $\ms C$ be a stable infinity category. 
Then the induced maps 
\[
\xymatrix{
\Stabr {\ho {\ms C}}	\ar@<1ex>[r]|-{d_0^*} \ar@<-1ex>[r]|(.45){d_1^*}	&	\Stabr {\ho {\ms C^{\Delta^1}}}
}
\]
satisfy the following:
\begin{enumerate}
\item $d_0^*$ and $d_1^*$ are both injective. 
\item The image $d_0^*$ does not intersect the image of $d_{1}^{*}$. 
\end{enumerate}
\end{prop}

\begin{proof}
Since the proof is similar, we only prove the assertion (1) for $d_0^*$.

Let $\sigma=(Z, \mca P)$ and $\tau = (W, \mca Q)$ be in $\Stabr {\ho{\ms C}}$, and suppose $d_0^* \sigma = d_0^* \tau$. 
Then they have the same central charge and this implies $Z= W$. 
It is enough to show $\mca P= \mca Q$ for the conclusion. 
Recall the semiorthogonal decomposition $\ho{\ms C^{\Delta^1}}= \< \mb D_{0}^{L}, \mb D_{0}^{R} \>$ in Lemma \ref{lm:SOD}. 
Since $\mb D_{0}^{L}=\mb C_{s}$, Proposition \ref{CP2.2} implies 
\[
\mca P(\phi) = d_0^* \mca P(\phi) \cap \mb C_{s} = d_0^* \mca Q(\phi) \cap \mb C_{s} = \mca Q(\phi). 
\]
This gives the proof of the first assertion.

Let 
$\ho{\ms C^{\Delta^1}}=\<  \mb D_0^L, \mb D_0^R  \>$ (resp. $\< \mb D_1 ^L, \mb D_1 ^R  \>$) be the semiorthogonal decomposition in Lemma \ref{lm:SOD}. 
Recall $\mb D_0^L=\mb D_1^R =\mb C_{s}$. 
Take $d_0^* \sigma \in \Im d_0^* $ and $d_1^* \tau \in \Im d_1^* $. 
Suppose that $d_0^* \sigma = d_1^* \tau$ for some $\sigma=(Z, \mca P) $ and $\tau = (W, \mca Q) \in \Stabr {\ho {\ms C}}$. 
Lemma \ref{lm:CP} implies  
\[
\mca P(\phi) = d_0^*\mca P \cap \mb D_0^L = d_1^* \mca Q \cap \mb D_0^L = d_1 ^* \mca Q \cap \mb D_1^R = \mca Q(\phi). 
\]
Furthermore we see 
\[
Z = d_0^* Z | _{\mb D^L_0} = d_1^* W |_{\mb D^L_0}= d_1 ^* W | _{\mb D^R_1} = W. 
\]
Hence $Z$ is the same as $W$. 
For $[f \colon x \to y] \in \ho {\ms C ^{\Delta^1}}$, we have 
\[
d_0^* Z (f) = Z(y)-Z(\fib f)\mbox{ and }d_1^*W(f)=W(x)-W(\cof f) . 
\]
Since $\fib f [1]=\cof f$, we see
$d_1^*W(f) = W(x) + W(\fib f) = W(y) + 2 W(\fib f)$. 
This gives a contradiction. 
\end{proof}

\subsection{Fullness of $d_0^* \sigma$ and $d_1^* \sigma$  }

The aim is to prove the fullness of $d_0^* \sigma$ and of $d_1^*\sigma$ are equivalent to that of $\sigma$. The following is a key ingredient in the proof. 
 
\begin{prop}\label{object-stable}
Let $\ms C$ be a stable infinity category and let $\sigma=(\mca A, Z)$ be a reasonable stability condition on $\ho {\ms C}$. 
\begin{enumerate}
\item[$(1)$] Suppose that $f \in \ms C^{\Delta^1}$ is semistable in $d_0^* \sigma $ with the phase $\phi$. Then we have
\begin{enumerate}
\item[$(1a)$] $d_0 f = y$ is $\sigma $-semistable, 
\item[$(1b)$] $\cof f=\fib f[1]$ is $\sigma $-semistable, and 
\item[$(1c)$] Moreover $\arg Z(y)=\arg Z(\cof f)= \phi$. 
\end{enumerate}
\item[$(2)$] Suppose that $f \in \ms C^{\Delta^1}$ is semistable in $d_1^* \sigma $ with the phase $\phi$. Then we have
\begin{enumerate}
\item[$(2a)$] $d_1 f = x$ is $\sigma $-semistable, 
\item[$(2b)$] $\fib f=\cof f[-1]$ is $\sigma $-semistable, and 
\item[$(2c)$] Moreover $\arg Z(x)=\arg Z(\fib f)= \phi$. 
\end{enumerate}
\end{enumerate}
\end{prop}

\begin{proof}
Put $d_0 f=y, d_1 f =x$. 
The central charge of $d_0^* \sigma$ is denoted by $d_0^* Z$ and the heart of $d_0^* \sigma$ is denoted by $d_0^*\mca A$. 

We first show the assertion $(1a)$. 
We may assume that the phase $\phi $ is in $(0,1] \subset \bb R$ without loss of generality. 
Then the object $f$ is in the heart $d_0^* \mca A$. 
In particular both $y$ and $\cof f$ belong to the heart $\mca A$ by Lemma \ref{lm:4.6}.

Suppose to the contrary that $y$ is not $\sigma$-semistable. 
Then there is a subobject $a$ of $y$ such that  $a$ is $\sigma $-semistable with 
\begin{equation}
\arg Z(a) >\arg Z(y) \label{a>y}. 
\end{equation}

From the definition of $d_{0}^{*}\mca A$, we have the following short exact sequence in $d_{0}^{*}\mca A$: 
\[
\begin{CD}
0	@>>>	j_{!}(\fib f)	@>>>	f	@>>>	\1_y	@>>>	0. 
\end{CD}
\]
Since $f$ is semistable in $d_0^* \sigma$, we have 
\begin{equation}
\arg Z (\cof f)= \arg d_0 ^* Z (j_{!}(\fib f))	\leq \arg d_0^*Z(f) \leq \arg d_0^* Z(\mr{id}_y) =\arg Z(y). 	\label{ni}
\end{equation}

Let $g$ be the composition of morphisms $g \colon a \subset y \to \cof f$. 
Then we have a commutative diagram in $\ho{\ms C}$: 
\[
\xymatrix{
\ker g	\ar[r]\ar[d]_{h}	&	x	\ar[d]^{f}	\\
a	\ar[r]	&	y. 
}
\]
Then we have a morphism $\tau \colon h \to f$ in $\ho{\ms C^{\Delta^{1}}}$ by Lemma \ref{lm:difference}. 
Since $\cof h = \im g$ is a colimit in the infinity category $\ms C$, a morphism $\cof \tau \colon \cof h \to \cof f$ in $\ms C$ induces the monomorphism $\xymatrix{\im g \ar@{^{(}->}[r]& \cof f}$ in $\ho{\ms C}$. 
Since $d_{0}\tau \colon d_{0}h \to d_{0}f$ is also 
monomorphism in $\mca A$, the mapping cone $\cof (h \to f)$ in $\ho{\ms C^{\Delta^{1}}}$ is in $d_{0}^{*} \mca A$ by Lemma \ref{lm:4.6}. 
Thus $[h \colon \ker g \to a]$ is a subobject of $f$. 
Since $f$ is $d_0^*\sigma$-semistable we see 
\begin{equation}
\arg d_0^*Z(h) \leq \arg d_0^*Z(f). 	\label{ichi}
\end{equation}

Since the morphism $a\to \im g$ is an epimorphism in $\mca A$ from the semistable object $a$  we see $\arg Z(a) \leq \arg Z(\im g)$. 
Moreover the definition of $d_0^*Z$ implies 
\begin{equation}
\arg Z (a) \leq \arg d_0^*Z(h) \leq \arg Z(\im g). 	\label{a<h}
\end{equation}
Hence the inequalities (\ref{a>y}), (\ref{a<h}) and (\ref{ni}) imply the following inequality 
\[
\arg d_0^*Z(f)	\leq \arg Z(y) < \arg Z(a) \leq \arg d_0^* Z(h)
\]
which contradicts (\ref{ichi}). 
Hence $y$ is $\sigma$-semistable.

For the proof of (1b), 
suppose to the contrary that $\cof f$ is not $\sigma$-semistable. 
Then there exists a quotient $a'$ of $\cof f$ in $\mca A$ such that $a'$ is $\sigma $-semistable with 
\begin{equation}
	\arg Z(a')	<	\arg Z(\cof f). 	\label{a'<cof}
\end{equation}
The inequality (\ref{ni}) implies 
\begin{equation*}
\arg Z(a')	<	\arg Z(y). 	
\end{equation*}
Since $y$ is semistable, we have $\Hom_{\ho{\ms C}}(y, a')=0$. 
Thus we obtain the following diagram of distinguished triangles in $\ho {\ms C}$ by $3 \times 3$ Lemma:
\[
\xymatrix{
z\ar[r]\ar[d]		&	x	\ar[r]\ar[d]_f	&	a'[-1]	\ar[d]_{h'}\\
y	\ar[r]\ar[d]		&	y	\ar[r]\ar[d]	&	0	\ar[d]\\
d	\ar[r]		&	\cof f	\ar[r]		&	a'
}
\]
In the diagram above, $d$ is the kernel of the morphism $\cof f \to a'$ in $\mca A$. 
By Lemma \ref{lm:difference}, the diagram above gives a morphism $f \to h'$ in $\ho{\ms C^{\Delta^{1}}}$. 
Since $\fib(f \to h')$ is isomorphic to $[z \to y]$, the object $\fib(f \to h')$ is in $d_{0}^{*}\mca A$ by Lemma \ref{lm:4.6} and we see that $h'$ is a quotient of $f$ in $d_{0}^{*}\mca A$. 
The following inequality holds since $f$ is $d_0^*\sigma$-semistable:
\begin{equation*}
\arg d_0^* Z(f)	\leq \arg d_0^* Z(h')  =\arg Z(a'). 
\end{equation*}
On the other hand the above inequality contradicts (\ref{ni}) and (\ref{a'<cof}).  
Hence $\cof f$ is semistable.

Finally we prove (1c). 
If $\arg Z(\cof f)\neq  \arg Z(y)$ then we have $\Hom_{\ho {\ms C}}(y, \cof f)=0$ by the inequality (\ref{ni}).  
Hence we have $\Hom_{\ho {\ms C^{\Delta^1}}} \left(\1_y, [\cof f \to 0]  \right)=0$ and this implies 
\[
f	\cong \1_y	\+		\left[ (\cof f)[-1]	\to 0 \right]. 
\]
Since $f$ is semistable $(\cof f)[-1]  =\fib f  =0$ or $y = 0$. 

The same argument for the opposite infinity category $\ms C^{\mr{op}}$ of $\ms C$ implies the second part $(2)$. 
\end{proof}

\begin{prop}\label{pr:semistable}
Let $\sigma =(\mca A, Z)$ be a reasonable stability condition on $\ho {\ms C}$. 

\begin{enumerate}
\item For an object $f \in \ms C^{\Delta^1}$, $f$ is $d_0^* \sigma$-semistable in phase $\phi$ if and only if $d_0 f$ and $\cof f$ are $\sigma$-semistable in phase $\phi$. 
\item For an object $f \in \ms C^{\Delta^1}$, $f$ is $d_1^* \sigma$-semistable in phase $\phi$ if and only if $d_1 f$ and $\fib f$ are $\sigma$-semistable in phase $\phi$. 
\end{enumerate}
\end{prop}

\begin{proof}
``Only if " part follows from Proposition \ref{object-stable}. 
``If" part is a direct consequence of Lemma \ref{lm:CP} as follows. 
Lemma \ref{lm:CP} implies both $s(y)$ and $j_!(\fib f)$ are $d_0^* \sigma$-semistable. 
Since $f$ is given by the extension of $s(y)$ and $j_!(\fib f)$, $f$ is $d_0^*\sigma$-semistable. 
This give the proof of the first assertion and the proof of the second assertion is similar. 
\end{proof}

\begin{prop}\label{pr:support}
Let $\sigma =(\mca A, Z)$ be a reasonable stability condition on $\ho {\ms C}$. 
The stability condition $\sigma$ satisfies the support property, if and only if $d_0^* \sigma$ (resp. $d_1^*\sigma$) satisfies the support property. 
\end{prop}

\begin{proof}
We only prove the assertion for $d_0^*$  since the proof is similar. 

Fix a norm $\|  - \|$ on $K_0(\ho {\ms C}) \otimes \bb R$. 
The Grothendieck group $K_0(\ho {\ms C^{\Delta^1}})$ is isomorphic to $K_0(\ho {\ms C})^{\+ 2}$ since $\ho {\ms C^{\Delta^1}}$ has a semiorthogonal decomposition by $\ho {\ms C}$. 
Consider the semiorthogonal decomposition $\ho{\ms C^{\Delta^{1}}} = \<\mb D_{0}^{L}, \mb D_{0}^{R}\>$.  
Let $\| - \|_{d_{0}}$ be a norm induced from the norm $\| -\|$ on $K_{0}(\ho{\ms C})\otimes \bb R$, that is, $\|  f \|_{d_{0}}:=\sqrt{\| d_{0}f \|^{2}+ \|  \cof f \|^{2}}$.

By Proposition \ref{pr:semistable}, $f \in \ms C^{\Delta^1}$ is $d_0^* \sigma$-semistable if and only if $d_0f $ and $\cof f $ are $\sigma$-semistable in the same phase. 
 Hence a $d_0^* \sigma$-semistable object $f$ satisfies 
\begin{align*}
|d_0^*Z(f)|  &=    |Z(d_{0}f) + Z(\cof f) | 	\\
			&=	|Z(d_{0}f)  | +|Z(\cof f)|	\\
			&\geq  \frac{1}{C}(	\|	d_{0}f	\|	+	\|	\cof  f\| ) \\
			&\geq  \frac{1}{C} \|f\|_{d_{0}}. 
			\end{align*}
In the first inequality above we use the support property for $\sigma$. 
Hence $d_0^* \sigma$ satisfies the support property.

Conversely suppose that $d_0^* \sigma$ satisfies the support property. 
Then there is a constant $C'$ such that 
$C' |Z(f)| \geq  \|f\|_{d_{0}}$ for any $d_0^* \sigma$-semistable object in $\ho{\ms C^{\Delta^{1}}}$. 
Take $f$ as the image $s(x)$ of $s \colon \ms C \to \ms C^{\Delta^1}$ where $x \in \ms C$ is $\sigma$-semistable. 
Since $s(x)$ is also $d_0^* \sigma$-semistable and the fact $Z(x)=d_0^*Z\left(s(x)\right)$, 
we have $C'|d_0^* Z\left(s(x)\right)| = C'|Z(x)| \geq  \| s(x) \|_{d_{0}} =   \|x \|$. 
Hence $\sigma$ satisfies the support property. 
\end{proof}

\begin{thm}\label{thm:full}
Let $\sigma$ be a reasonable stability condition on $\ho {\ms  C}$. 
Then the following are equivalent:
\begin{center}
$(1)$ $\sigma$ is full, $(2)$ $d_0^* \sigma$ is full, and $(3)$ $d_1^* \sigma$ is full. 
\end{center}
\end{thm}

\begin{proof}
By Proposition \ref{pr:support} both $d_1^* \sigma$ and $d_0^* \sigma$ satisfy the support property if so does $\sigma$. 
By Lemma \ref{BMb.4}, the support property on a stability condition is equivalent to the fullness of it. 
Thus we complete the proof. 
\end{proof}

\subsection{On the images $\Im d_0^*$ and $\Im d_1^*$}
Now we prove that both images $\Im d_0^*$ and $\Im d_1^*$ are closed in $\Stabr {\ho {\ms C^{\Delta^1}}}$ by using the ``inducing" construction due to \cite{MR2524593}. 
Let us recall the construction. 

Let $F \colon \mb D \to \mb D'$ be an exact functor between triangulated categories. 
Assume that $F$ satisfies the following additional condition
\begin{enumerate}
\item[(Ind)] $\Hom_{\mb D'}(F(a) , F(b))=0$ implies $\Hom_{\mb D}(a,b)=0$ for any $a, b\in \mb D$. 
\end{enumerate}

\begin{rmk}
Recall functors  $j_!$ and $ j_*\colon \ho {\ms C} \to \ho {\ms C^{\Delta^1}}$ in (\ref{eq:inverse}). 
Then three functors $s$, $j_!$ and $j_*$ from $\ho  {\ms C}$ to $\ho {\ms C^{\Delta^1}}$ satisfy the condition (Ind) since they are faithful. 
\end{rmk}

Let $\sigma' =(Z', \mca P')\in \Stab {\mb D'}$. 
Define $F^{-1}\sigma'$ by the pair $(Z, \mca P)$ where 
\[
Z = Z'\circ F, \ \mca P(\phi) = \{ x \in \mb D \mid F(x) \in \mca P'(\phi) \}. 
\]
Then one can easily see that the pair $F^{-1} \sigma'$ is a stability condition on $\mb D$ if and only if $F^{-1} \sigma' $ has the Harder-Narasimhan property. 

\begin{lem}[{\cite[Lemmas 2.8 and 2.9]{MR2524593}}]\label{lm:MMS}
Notation is the same as above. 
\begin{enumerate}
\item Then the set 
\[
\Dom F =\{	\sigma ' \in \Stab{ \mb D'} \mid F^{-1}	 \sigma'  \in \Stab {\mb D} \}
\]
is closed in $\Stab {\mb D'}$. 
\item The map $F^{-1} \colon \Dom F \to \Stab {\mb D}$ is continuous. 
\end{enumerate}
\end{lem}

\begin{rmk}
We note that $F^{-1} \sigma'$ is reasonable if $\sigma \in \Stab {\mb D'}$ is reasonable. 
Hence the restriction of $F^{-1}$ to reasonable stability conditions is well-defined. 
\end{rmk}

By using Lemma \ref{lm:MMS} we show that both images $\Im d_0^* $ and $\Im d_1^*$ are closed. 

\begin{thm}\label{thm:closed}
Let $\ms C$ be a stable infinity category. 
Then 
\begin{enumerate}
\item Images $\Im d_0^*$ and $\Im d_1^* $ are given by 
\begin{align}
\Im d_0^* &=	\{ \sigma \in \Stabr {\ho {\ms C^{\Delta^1}}} \mid 	\sigma \in \Dom s \cap \Dom {j_!},\   s^{-1} \sigma = [1] \cdot j_!^{-1} \sigma		\}	\label{000} \mbox{ and }\\ 
\Im d_1^* &=	\{ \sigma \in \Stabr {\ho {\ms C^{\Delta^1}}} \mid 	\sigma \in \Dom s \cap \Dom {j_*},\   s^{-1} \sigma = [-1] \cdot j_*^{-1} \sigma		\}. 	\notag
\end{align}
\item Both $\Im d_0^*$ and $\Im d_1^*$ are closed. 
\end{enumerate}
\end{thm}

\begin{proof}
Since the proof is similar we only prove the assertion for $d_0^*$. 

Let $\tau$ be in $\Stabr{\ho{\ms C}} $ and put $\sigma=d_{0}^{*}\tau $. 
Take the Harder-Narasimhan filtration of an object $x \in \ho {\ms C}$ with respect to $\tau$:
\[
\HNF{x}{x}{a}
\] 
Then $s(a_i)$ is $d_0^* \tau$-semistable by Lemma \ref{lm:CP}. 
Since the Harder-Narasimhan filtration is unique, the value of the filtration by the functor $s \colon \ho{\ms C} \to \ho{\ms C^{\Delta^1}}$ gives the Harder-Narasimhan filtration of $s(x)$ with respect to $d_0^* \tau$. 
Thus $s^{-1} \sigma$ has the Harder-Narasimhan property and $\sigma \in \Dom s$. 

The same argument for $j_!$ implies that $\Im d_0^*$ is a subset of $\Dom {j_!}$. 
The condition $\   s^{-1} \sigma = [1] \cdot j_!^{-1} \sigma$ is obvious. 
Thus $\Im d_0^*$ is a subset of the right hand side in (\ref{000}). 

Conversely take $\sigma  =(Z, \mca P)$ from the right hand side in (\ref{000}). 
Put $\tau  = (W, \mca Q) = s^{-1} \sigma = [1]j_! ^{-1} \sigma$. 
We wish to prove $d_0^* \tau = \sigma$. 

Any object $f \in \ho {\ms C^{\Delta^1}}$ has the following decomposition:
\begin{equation}
\xymatrix{
j_{!} \circ  \fib (f) \ar[r]		&	f	\ar[r]	&	s \circ d_0(f) \ar[r]	&	j_{!} \circ  \fib (f)[1].  
}\label{111}
\end{equation}
Hence we have $Z(f) = Z\left(j_{!} \circ  \fib (f)	\right) + Z\left( s \circ d_0 (f)\right)$. 
Moreover we see $Z\left( s \circ d_0 (f)\right) = W(d_{0}f)$ since $W$ is the central charge of $s^{-1}\sigma$. 
Similarly we have $ Z\left(j_{!} \circ  \fib (f)	\right)  = W(\fib f[1]) $. 
Thus we see 
\[
Z(f) = Z\left(	j_{!} \circ  \fib (f)	\right) + Z\left( s \circ d_0 (f)\right)  = W(\fib f[1]) + W\left(s \circ d_0(f)\right) = d_0^* W (f). 
\]
Hence $d_0^* \tau$ and $\sigma$ have the same central charge. 

Thus it is enough to show that $d_0^* \tau$ and $\sigma$ have the same heart. 
Take an object $f \in d_0^* \mca Q(\phi)$. 
Then Proposition \ref{object-stable} implies $\cof f$ and $d_0 f$ are both $\tau$-semistable in phase $\phi$. 
Since $\tau $ is $s^{-1} \sigma$, the object $s \circ d_0 (f) $ is in $\mca P(\phi)$. 
Similarly the equality $\tau = [1] j_!^{-1} \sigma$ implies that $j_! ( \cof f)$ is in $\mca P(\phi +1)$. 
In particular we see $j_{!} (  \fib f) \in \mca P(\phi)$.

Thus $s \circ d_0(f)$ and $j_{!} \circ  \fib (f)$ are both $\sigma$-semistable in the phase $\phi$ if $f \in d_0^* \mca Q(\phi)$. 
Hence $f$ is $\sigma$-semistable in phase $\phi$ by (\ref{111}) and we see $d_0^* \mca Q(\phi) \subset \mca P(\phi)$.
Then the heart of the $t$-structure of $d_0^* \tau$ is contained in the heart of $\sigma$. 
Then both hearts should be the same since they are hearts of bounded $t$-structures on $\ho {\ms C^{\Delta^1}}$. 
Thus we prove the assertion (\ref{000}).

Define $\pi \colon \Stabr {\ho {\ms C^{\Delta^1}}} \to \Stabr {\ho {\ms C}} \times \Stabr {\ho {\ms C}}$ by $\pi (\sigma)=\left(s^{-1} \sigma,  [1]\cdot  j_! ^{-1}\sigma \right)$. 
Then $\pi$ is continuous by Lemma \ref{lm:MMS}. 
The right hand side in (\ref{000}) is the inverse image of the diagonal $\Delta_{\mathsf{Stab}}$ by the map $\pi$. 
Since the diagonal is closed, $\pi^{-1}\Delta_{\mathsf{Stab}} = \Im d_0^*$ is closed. 
\end{proof}

\section{An example of $\ms C^{\Delta^1}$}

The aim of this section is the proof of Theorem \ref{thm-main2}. 

\begin{dfn}[Construction of path]\label{dfn:construction}
Put $g_{\theta} \in \widetilde{\mr{GL}}_2(\bb R)$ by 
\begin{equation}\label{eq:rotation}
g_{\theta}=\left(\exp(\sqrt{-1}\pi \theta), f(t)=t+\theta \right); \theta \in [0,1). 
\end{equation}
Suppose that a reasonable stability condition $\sigma =(\mca A, Z)$ on $\ho {\ms C}$ satisfies the following condition: 
\begin{itemize}
\item[(Deg)] The image of $Z \colon K_{0}(\ho{\ms C} ) \to \bb C$ is contained in the subset $\bb R$ of real numbers in $\bb C$. 
\end{itemize}

\begin{enumerate}
\item 
Let $\ho {\ms C^{\Delta^1}} = \< \mb D_{0}^{L}, \mb D_{0}^{R} \>$ be the semiorthogonal decomposition in Lemma \ref{lm:SOD}. 
We define $d_0^*(\theta) \sigma$ by $\gl {\sigma} {[-1]\sigma g_{\theta}} =\gl {\sigma} {\sigma g_{-1+\theta}} $. 
The central charge and the heart  of $d_0^*(\theta)\sigma$ are respectively denoted by $Z_0^{\theta}$ and $ \mca A_0^{\theta}$. 
\item 
Let $\ho {\ms C^{\Delta^1}} = \< \mb D_{1}^{L}, \mb D_{1}^{R} \>$ be the semiorthogonal decomposition in Lemma \ref{lm:SOD}. 
We define $d_1^*(\theta) \sigma$ by $\gl {[1]\sigma} {\sigma g_{\theta}}  $. 
The central charge and the heart  of $d_1^*(\theta)\sigma$ are respectively denoted by $Z_1^{\theta}$ and $ \mca A_1^{\theta}$. 
\end{enumerate}
\end{dfn}

\begin{lem}\label{lm:degenerate}
Let $\sigma$ be a stability condition on $\ho{\ms C}$. 
If $\sigma$ satisfies the condition $(\mr{Deg})$ then any object in the heart of $\sigma$ is $\sigma$-semistable. 
\end{lem}

\begin{proof}
The condition (Deg) implies that any object in $\mca A$ has phase $1$. 
Then any object in $\mca A$ is $\sigma$-semistable by the definition of stability conditions. 
\end{proof}

\begin{rmk}
If $\theta=0$, then $d_0^*(0) \sigma$ and $d_1^*(0) \sigma$ are the same as respectively $d_0^* \sigma$ and $d_1^* \sigma$. 

The condition (Deg) is necessary since it guarantees the conditions (1) and (2) in Proposition \ref{CP2.2}. 
We note that these conditions fail if $\theta  = 1$. 
\end{rmk}

\begin{lem}\label{lm:semistable}
Suppose that a reasonable stability condition $\sigma =(\mca A, Z) \in \Stab {\ho {\ms C}}$ satisfies the condition {\rm (Deg)}. 
\begin{enumerate}
\item 
Then any object $j_{*}(x)=[0 \to x] $ for $x \in \mca A$ is $d_0^*(\theta)\sigma$-semistable for any $\theta \in [0,1)$. 
\item 
Then any object $j_{!}(x)=[x \to 0] $ for $x \in \mca A$ is $d_1^*(\theta)\sigma$-semistable for any $\theta \in [0,1)$. 
\end{enumerate}
\end{lem}

\begin{proof}
Since the proof is similar, we only prove the first assertion. 

We first note that the heart $\mca A_0^{\theta}$ is constant for any $\theta \in [0,1)$, 
since the stability condition $\sigma=(\mca A, Z)$ on $\ho {\ms C}$ satisfies the condition {\rm (Deg)}.  
If $ x \in \mca A$, Lemma \ref{lm:4.6} implies that the object $j_{*}(x)\in \ms C^{\Delta^1}$ is in $\mca A_{0}^{\theta}$ since $\mca A_{0}^{\theta}$ is constant. 

We can see that an object $[g \colon y \to z] \in \mca A_0^{\theta}$ is a subobject of $j_{*}(x)$ 
if and only if  $z \subset \cof g \subset x \in \mca A$ by $3 \times 3$ lemma and Lemma \ref{lm:difference}. 
Moreover $y$ is given by $(\cof g/z)[-1]$ by if $g \subset  j_{*}(x)$. 
The semiorthogonal decomposition $\ho {\ms C^{\Delta^1}}  = \< \mb D_{0}^{L}, \mb D_{0}^{R}  \>$ implies the following exact sequence: 
\[
\xymatrix{
\cof g [-1]	\ar[r]\ar[d]_h	&	(\cof g /z)[-1]\ar[r]\ar[d]_g	&	z\ar[d]_{\1}	\\
0	\ar[r]	&	z	\ar[r]	&	z\\
}
\]

Put $Z(z)=-\beta, Z(x) = - \alpha , Z(\cof g)=- \gamma $ where $\alpha, \beta$ and $\gamma $ are positive real numbers. 
Then we have 
\begin{align*}
Z_0^{\theta}(g) &= Z_0^{\theta}(h) + Z_0^{\theta}(\1 _z)= -\beta  - \gamma \cdot \exp(-\pi \sqrt{-1}\theta) 	\\
Z_0^{\theta}(j_{*}(x)) &= Z_0^{\theta}(\1_x) + Z_0^{\theta}(x[-1] \to 0 ) = -\alpha -\alpha \exp(-\pi \sqrt{-1}\theta). 
\end{align*}
Since $z \subset \cof g \subset x$, we have $\beta \leq \gamma \leq \alpha$ which simply implies the inequality 
$\arg Z_0^{\theta}(g) \leq  \arg Z_0^{\theta}(j_{*}(x) )$ as follows.  
In fact it is enough to show that 
the imaginary part of $Z_{0}^{\theta}(j_{*}(x))  \overline{Z_{0}^{\theta}(g)}$ is non-negative where $\overline{Z_{0}^{\theta}(g)}$ is a complex conjugate of $Z_{0}^{\theta}(g)$. 
Since the imaginary part is given by $\alpha(\gamma -\beta )\sin(\pi \theta)$, we finish the proof. 
\end{proof}

\begin{lem}\label{lm:classification}
Let $d_0^* (\theta) \sigma$ and $d_1^*(\theta)\sigma$ be the stability conditions on $\ho{\ms C^{\Delta^1}}$ defined in Definition \ref{dfn:construction}. 
Assume $\theta \neq 0$. 
\begin{enumerate}
\item 
Any object $[f \colon x \to y] \in \ho {\ms C^{\Delta^1}}$ in the heart $\mca A_0^{\theta}$ has the Harder-Narasimhan filtration as follows:
\begin{equation}\label{eq:HNfiltration}
\xymatrix{
[\ker \delta  \stackrel{\1}{\to} \ker \delta ]	\ar[r]	&	[\ker \delta \to y]		\ar[r]\ar[d]	 & [x \stackrel{f}{\to} y] \ar[d]	\\
		& [0 \to \im \delta ]	\ar@{-->}[lu] 	&	[\cok \delta [-1] \to 0],  \ar@{-->}[ul]\\
}
\end{equation}
where $\im \delta  , \ker \delta , \cok \delta $ are respectively the image, the kernel and the cokernel of the morphism $[\delta \colon y \to \cof f ]$. 
\item 
Any object $[f \colon x \to y] \in \ho {\ms C^{\Delta^1}}$ in the heart $\mca A_1^{\theta}$ has the Harder-Narasimhan filtration as follows:
\[
\xymatrix{
\left[0 \to  \ker \epsilon[1] \right]	\ar[r]	&	\left[\im \epsilon \to \ker \epsilon[1] \right]		\ar[r]\ar[d]	 & [x \stackrel{f}{\to} y] \ar[d]	\\
		& [\im \epsilon \to 0 ]	\ar@{-->}[lu] 	&	[\cok \epsilon  \stackrel{\1}{\to} \cok \epsilon],  \ar@{-->}[ul]\\
}
\]
where $\im \epsilon  , \ker \epsilon , \cok \epsilon $ are respectively the image, the kernel and the cokernel of the morphism $[\epsilon \colon \fib f \to x]$. 
\end{enumerate}
\end{lem}

\begin{proof}
We only prove the first assertion by the same reason in Lemma \ref{lm:semistable}. 

There is a distinguished triangle in $\ho {\ms C}$:
\[
\xymatrix{
\ker \delta	\ar[r]	&	x	\ar[r]	& \cok \delta[-1]	
}
\]

The construction of $d_0^*(\theta) \sigma$ implies that $\1_{\ker \delta }$ is semistable in $d_0^* (\theta) \sigma$ with phase $1$ 
and that $[\cok \delta [-1] \to 0]$ is $d_0^* (\theta) \sigma$-semistable in phase $1-\theta$ by Lemma \ref{lm:degenerate}. 

The following diagram and Lemma \ref{lm:difference}
\[
\xymatrix{
\ker \delta \ar[r]\ar[d]_{f_1}	& x \ar[r]\ar[d]_f	&	\cok \delta[-1]	\ar[d]_{\alpha_2}	\\
y \ar[r]	&		y	\ar[r]	&	0 	\\
}
\]
implies that $\alpha _2$ is a quotient of $f$ by Lemma \ref{lm:4.6}, since $\mca A_{0}^{\theta}$ is constant for $\theta$. 

Similarly, by the following diagram,  
\[
\xymatrix{
\ker \delta \ar[r]\ar[d]_{\1}	&	\ker \delta	\ar[r]\ar[d]_{f_1}	&	0\ar[d]_{\alpha _1}	\\
\ker \delta	\ar[r]	&	y	\ar[r]	&	\im \delta		\\
}
\]
we see that $\1 _{\ker }$ is a subobject of $f_1$ by Lemma \ref{lm:4.6}.
Thus we obtain a filtration denoted in (\ref{eq:HNfiltration}). 
By Lemma \ref{lm:semistable}, the morphism $[0 \to \im \delta]$ is $d_0^*(\theta) \sigma$-semistable. 
If $\theta \neq 0$ then the phase of $[ 0 \to \im \delta]$ is smaller than $1$ and is bigger than $1-\theta$. 
Thus the filtration in (\ref{eq:HNfiltration}) gives an HN filtration of $f$. 
\end{proof}

\begin{cor}\label{cr:torsionpiar}
Let $d_0^* (\theta) \sigma$ be the stability condition constructed in Definition \ref{dfn:construction}. 
Suppose that $\theta \neq 0$. 
\begin{enumerate}
\item Any semistable object in $d_0^* (\theta)\sigma$ is one of the following:
\[
\left\{
[x	\stackrel{\1}{\to}	x], 
\left[y[-1]	\to 	0\right], 
[0	\to	z]
 \middle|
 x,y,z\in \mca A\mbox{ where $\sigma=(\mca A, Z)$}
\right\}
\]
\item 
The pair $(\mca T, \mca F)$ gives a torsion pair on the heart $\mca A_0^{\theta}$ where $\mca T$ and $\mca F$ are respectively 
\[
\mca T =	\left\{ 
[x	\stackrel{\1}{\to} 	x]  \in \mca A_0^{\theta}	\middle|	 x\in \mca A	 \right\} \mbox{ and }
\mca F = \left\{ 
\left [y \to z  \right]
\in \mca A_0^{\theta}	\middle|  y \in \mca A[-1], z \in \mca A		\right\}	. 
\]
\end{enumerate}
\end{cor}

\begin{proof}
The first assertion is obvious from Lemma \ref{lm:classification}. 
For the second assertion, 
take any morphism $[f \colon x \to y ] \in \mca A_0^{\theta}$ and let $\delta$ be the morphism $y \to \cof f$. 
Since $\1_{\ker \delta}$ is a subobject of $f$, we denote by $e$ the quotient $f/\1_{\ker \delta}$. 
Then $d_{0}e$ is isomorphic to $\im \delta$ and $d_{1}e$ is isomorphic to $\cok \delta[-1]$. 
Thus $e$ is given by $[e \colon \cok \delta[-1] \to \im \delta]$. 
Then the object $[e \colon \cok \delta[-1] \to \im \delta]$ clearly is in $\mca F$. 
The adjunction $ s \dashv d_1$ implies $\Hom_{\ho{ \ms C^{\Delta^1}}}(\mca T, \mca F)=0$. 
Thus we conclude the proof. 
\end{proof}

\begin{thm}\label{thm:P^1}
Suppose that an infinity category $\ms C$ is the infinity category $\ms D_{\mr{coh}}^{b}(\bb P^1) $ of $\bb P^1$ (cf. Definition \ref{dfn:derivedmugen}). 
Then two distinct images $d_1^*, d_0^* \colon \Stab (\ho{ \ms C}) \to \Stab (\ho {\ms C^{\Delta^1}})$ are disjoint but path connected to each other. 
\end{thm}

\begin{proof}
Let $Q$ be a quiver denoted by $\xymatrix{\mb v_1 \ar@<0.5ex>[r] \ar@<-0.5ex>[r]  & \mb v_2 } $ and let $R$ be the path algebra of $Q$ over $\mb k$. 
The algebra $R$ is described by exceptional collections $\{ \mca O, \mca O(1) \}$ in $\ho{\ms C}$; $R \cong \mr{End}_{\bb P^1} (\mca O \+ \mca O(1))$. 
The derived category $\ho{\ms C}$ is equivalent to the bounded derived category $\mb D^{b} (\mb{mod}\, R)$ of $R$ with the functor 
\[
\bb R\Hom _{\bb P^1}(\mca O \+ \mca O(1), - ) \colon \ho{\ms C} \to \mb D^{b} (\mb{mod}\, R)
\] 

Take a stability condition $\sigma=(\mca A , Z)$ on $\ho{\ms C}$ such that any simple object in the abelian category $\mb{mod}\, R$ is stable in phase $1$ \footnote{Equivalently take a stability condition $\sigma=(\mca A , Z)$ on $\mb D^{b}(\mb{coh}\, \bb P^1)$ such that $\mca O$ and $\mca O(-1)[1]$ are both $\sigma$-stable with $
Z(\mca O)=Z(\mca O(-1)[1])=-1. $
}. 
Then $\sigma$ satisfies the degenerate condition {\rm (Deg)} in Definition \ref{dfn:construction}. 
Moreover $\Stab {\ho{\ms C}}$ is connected by \cite{MR2335991} or \cite{MR2219846}. 
Thus it is enough to show that $d_0^* \sigma$ and $d_1^* \sigma$ are path connected.

\textbf{Step 1}. 
By Definition \ref{dfn:construction}, there is a collection $\{ d_0^* (\theta)\sigma \}_{0 \leq \theta \leq 2/3}$ of stability conditions on $\ho{\ms C^{\Delta^{1}}}$. 
Since $\sigma$ is reasonable, 
the collection $\{ d_0^* (\theta)\sigma \}_{0 \leq \theta \leq 2/3}$ is a continuous family in $\Stab {\ho{\ms C^{\Delta^{1}}}}$. 
Thus the collection is a path in $\Stab {\ho{\ms C^{\Delta^{1}}} }$

\textbf{Step 2}. 
Let $x,y$ and $z$ are in the heart $\mca A$ and set $\theta=2/3$. 
By the construction of $d_0^* (\theta)\sigma$, objects $[\1 \colon x \to x]$ and $\left[y[-1] \to 0\right]$ are  $d_0^*(2/3) \sigma$-semistable in phase respectively $1$ and $1/3$. 
By Lemma \ref{lm:semistable}, the object $[0 \to z] \in \mca A_0^{2/3}$ is semistable in phase $2/3$. 
Moreover 
the torsion pair $\pair{\mca P_0^{2/3}\left ((2/3,  1] \right)}{\mca P_0^{2/3}\left((0,2/3]  \right)}$ is just $\pair{\mca P_0^{2/3} (1)}{\mca P_0^{2/3}([1/3, 2/3])}$. 
Recall $g_{2/3} \in \widetilde{\mr{GL}}_2(\bb R) $ defined in (\ref{eq:rotation}). 
We denote $(d_0^* (2/3) \sigma ) \cdot g_{2/3}$  by $\tau$. 
Then the heart $\mca B$ of $\tau$ is $\mca P_0^{2/3}(2/3, 5/3]$ and any $\tau$-semistable objects in $\mca B$ is one of the following: 
\[
\left[\gamma \colon 0 \to z[1]\right], [\beta \colon y\to 0] \mbox{ and } [\1_x \colon x \to x]. 
\]
We note that the phases of $\gamma, \beta$ and $\1_x$ in $\tau$ are respectively $1$, $2/3$ and $1/3$. 

\textbf{Step 3}. 
Take the semiorthogonal decomposition $\ho{\ms C^{\Delta^{1}}}=\<\mb D_{1}^{L}, \mb D_{1}^{R}\>$. 
Similarly to the case of $d_0^*(\theta)\sigma$, the collection $\{ d_1^*(\theta) \sigma \}_{0 \leq \theta \leq 2/3}$ determines a path in $\Stab {\ho{\ms C^{\Delta^{1}}}}$. 
The second part of Lemma \ref{lm:classification} implies that 
any semistable object in $d_1^* (\sigma)$ is one of the following:
\[
\left\{
\left[0 	\to	x[1] \right], 
\left[y	\to 	0\right], 
[z \stackrel{\1_z}{\to}	z]
 \middle|
 x,y,z\in \mca A\mbox{ where $\sigma=(\mca A, Z)$}
\right\}. 
\]
Furthermore if $\theta = 2/3$ then the phases of $\left[0 \to	x[1] \right], [y \to 0] $ and $[z \stackrel{\1}{\to} z]$ are respectively $1$, $2/3$ and $1/3$. 
Hence $d_1^*(2/3)\sigma$ is the same as $\tau$ defined in Step 2. 
Thus we obtain a path connecting $d_1^* \sigma$ to $d_0^* \sigma$. 
\end{proof}

\begin{rmk}
The stability condition $\sigma$ taken in the proof above is not geometric. 
Namely skyscraper sheaves $\mca O_x$ is not $\sigma$-stable but $\sigma$-semistable. 
If the genus of a smooth projective curve is positive, such a non-geometric stability condition does not exist. 
\end{rmk}

The same argument in Theorem \ref{thm-main2} is effective for the distinguished full component $\Stabd {\mb D^{b}(\mb{coh}\, \bb P^2)}$ of the space of stability conditions on the bounded derived category $\mb D^{b}(\mb{coh}\, \bb P^2) \simeq \ho{\ms D_{\mr{coh}}^{b}(\bb P^{2})}$ of the projective plane $\bb P^2$. 
Namely we have

\begin{cor}\label{cor:Li}
Suppose an infinity category $\ms C$ is the infinity category $\ms D_{\mr{coh}}^{b}(\bb P^{2})$ (cf. Definition \ref{dfn:derivedmugen}).  
Let $\Stabd {\ho{\ms C}}$ be the distinguished full component of $\Stab {\ho{\ms C}} $ and 
let $d_0^*|_{\dagger}$ (resp. $d_1^*|_{\dagger}$) be the restriction of $d_0^*$ (resp. $d_1^*$) to $\Stabd {\ho{\ms C}}$. 
Then $\Im d_0^*|_{\dagger}$ is path connected to $\Im d_1^* | _{\dagger}$. 
\end{cor}

\begin{proof}
Due to Li \cite{MR3703470}, $\Stabd {\ho{\ms C}}$ is the union of algebraic stability conditions and geometric stability conditions. 
Since $\ho{\ms C} $ is equivalent to the bounded derived category of representations of the quiver $Q =\xymatrix{\bullet \ar@<1ex>[r] \ar[r] \ar@<-1ex>[r]  & \bullet \ar@<1ex>[r] \ar[r] \ar@<-1ex>[r]  & \bullet}$. 
Thus there exists $\sigma \in \Stabd {\ho{\ms C}}$ such that any simple module in the abelian category $\mb{mod}( \mb k Q)$ of the path algebra $\mb k Q$ is 
$\sigma$-stable in phase $1$. 
Since the stability condition $\sigma $ satisfies the condition (Deg), the same argument in Theorem \ref{thm:P^1} works. 
\end{proof}

\section{Another construction of $\ho{\ms C^{\Delta^{1}}}$}


Let $\mr {Mor}(\mca B)$ be the category of morphisms in an abelian category $\mca B$. 
The category $\mr{Mor}(\mca B)$ is tautologically the same as the category of functors from the ordinal $[1]=\{0 < 1	\}$ to $\mca B$. 
Since $\mr{Mor}(\mca B)$ is also an abelian category, 
one could define the bounded derived category $\mb D^{b}(\mr{Mor}(\mca B))$ by the localization of quasi-isomorphisms. 
When $\mca B$ is the abelian category $\mb{coh}(X)$ of coherent sheaves on $X$, 
we show that the derived category $\mb D^{b}(\mr{Mor}(\mca B))$ is equivalent to the homotopy category $\ho{\ms D_{\mr{coh}}^{b}(X)^{\Delta^{1}}}$ in Corollary \ref{pr:MMR} below.

\begin{prop}\label{pr:eqF}
Suppose that $\mca A$ is a Grothendieck abelian category. 
Then the unbounded derived category $\mb D (\mr{Mor}(\mca A))$ is equivalent to the homotopy category $\ho{\ms D(\mca A)^{\Delta^{1}}}	$ of the infinity category $\ms D(\mca A)^{\Delta^{1}}$
\end{prop}

\begin{proof}
Throughout the proof we identify $\K(\mr{Mor}(\mca A))$ with $\mr{Mor}(\K (\mca A))$. 
Since $\mca A$ is the Grothendieck abelian category, so is $\mr{Mor}(\mca A)$ by \cite{MR102537}. 
Thus the derived category $\mb D(\mr{Mor}(\mca A))$ is obtained by the homotopy category $\mr{Ho}(\K (\mr{Mor}(\mca A)))$ with respect to the injective model structure on $\K(\mr{Mor}(\mca A))$. 
Moreover $\K(\mca A)$ is also the injective model category by Proposition-Definition \ref{dfn:injetivemodel}. 

We wish to define a functor $\Phi \colon \mr{Ho}(\K(\mr{Mor}(\mca A))) \to \ho{ \ms D(\mca A)^{\Delta^{1}}	}$ which gives an equivalence. 
The proof consists of 3 steps. 
In Step 1, we define a functor $\Phi $ and show that $\Phi $ is essentially surjective. 
We next show that $\Phi$ is full in Step 2 and show that $\Phi$ is faithful in Step 3. 

\textbf{Step 1}. 
Let $f$ be in $\K(\mr{Mor}(\mca A))=\mr{Mor}(\K(\mca A))$. 
Then $f$ is given by the form $[f\colon x \to y]$ where $x $ and $y$ are in $\K(\mca A)$. 
There exist functors between $\K(\mr{Mor} (\mca A))$ and $\K(\mca A)$: 
\begin{itemize}
\item $j_{*} \colon \K(\mca A) \to \K(\mr{Mor}(\mca A)); j_{*}(z)=[0 \to z]$, 
\item $d_{0} \colon \K(\mr{Mor}(\mca A)) \to \K(\mca A); d_{0}([f \colon x \to y])=y$, 
\item $s \colon \K(\mca A) \to \K(\mr{Mor}(\mca A)); s(z)=[\1 \colon z \to z]$, and 
\item $d_{1} \colon \K(\mr{Mor}(\mca A)) \to \K(\mca A); d_{1}([f \colon x \to y])=x$. 
\end{itemize}
Moreover $j_{*}$ is the left adjoint of $d_{0}$ and $s$ is also the left adjoint of $d_{1}$. 
Since $j_{*}$ and $s$ preserve trivial cofibrations, both $d_{0}$ and $d_{1}$ preserve fibrations. 
%
%
%
Hence, if $f$ is fibrant in $\K(\mr{Mor}(\mca A))$, then both $d_{1}f=x $ and $d_{0}f=y$ are fibrant in $\K(\mca A)$. 
Hence $f$ gives  an object in $\ho{\ms D(\mca A)^{\Delta^{1}}}$ and we define $\Phi(f) $ by $f$ it self.

We next define $\Phi([\rho ])$ for a morphism $[\rho] \colon f \to g $ in the category $\mr{Ho}(\K(\mr{Mor}(\mca A)))$. 
Put $d_{1}g=z$ and $d_{0}g=w$. 
Recall that the morphism $[\rho]$ is a chain homotopy class of a morphism $\rho \colon f \to g$ in the category $\K(\mr{Mor}(\mca A))$ and that $\rho$ is given by a pair $(\rho_{1},\rho_{0} )$ of morphisms $ \rho_{1} \colon x \to y$ and $\rho_{0} \colon y \to w$ in $\K(\mca A)$ satisfying $\rho_{0}f = g \rho_{1}$. 
On the other hand, a morphism 
$\Phi (f) \to \Phi(g)$ 
in $\ho{\ms D(\mca A)^{\Delta^{1}}}$ is the equivalence class of a 5-tuple $\varphi =(\tau_{1}, \tau_{0}, \psi, h_{1}, h_{0})$ satisfying the following relations 
(see also the proof of Proposition \ref{pr:k-linear}): 
\begin{itemize}
\item $\tau_{1} \in  Z^{0}_{\mr{Ch}(\mca A)}(x, z)$, $\tau_{0} \in Z^{0}_{\mr{Ch}(\mca A)}(y,w)$ , $\psi \in Z^{0}_{\mr{Ch}(\mca A)}(x, w) $, 
\item  $h_{1} \in \Map_{\mr{Ch}(\mca A)}^{-1}(x, w)$  with $\delta (h_{1})=g\tau_{1}-\psi$, and 
\item  $h_{0} \in \Map_{\mr{Ch}(\mca A)}^{-1}(x, w)$  with $\delta( h_{0})=\tau_{0}f-\psi $. 
\end{itemize}

For the morphism $\rho \colon f \to g$ in $\K(\mr{Mor}(\mca A))$, let us denote by $\phi_{\rho}$ the $5$-tuple $(\rho_{1}, \rho_{0}, \rho_{0}f, 0,0)$ which gives a morphism in $\ms D(\mca A)^{\Delta^{1}}$. 
We wish to define $\Phi([\rho])$ by the equivalence class $[\phi_{\rho}]$. 
To complete the definition of $\Phi$, 
let $\tilde {\rho}= (\tilde{\rho}_{1}, \tilde {\rho}_{0})$ be a morphism in $\K(\mr{Mor}(\mca A))$ which is chain homotopic to the morphism $\rho$. 
Then there exists a morphism $\eta \in \Map_{\mr{Ch}(\mr{Mor}(\mca A))}^{-1}(f,g)$ such that $\delta (\eta )= \rho- \tilde {\rho}$. 
In particular the morphism $\eta$ is given by a pair $(\eta_{1}, \eta_{0})$ of morphisms $\eta_{1} \in \Map_{\mr{Ch}(\mca A)}^{-1}(x,z)$ and 
$\eta_{0} \in \Map_{\mr{Ch}(\mca A)}^{-1}(y,w)$ satisfying $\eta_{0}f=g \eta_{1}$ and 
 $\delta (\eta _{i})= \rho_{i} - \tilde {\rho}_{i}$ for $i \in \{0, 1\}$. 
By the relation $\eta_{0}f=g \eta_{1}$ and the morphism $\eta$, one can  find a $7$-tuple which gives an equivalence between $\phi_{\tilde {\rho}}$ and $\phi_{\rho}$. 
Thus we can define $\Phi([\rho])$ by the equivalence class $[\phi_{\rho}]$ and obtain the functor $\Phi$: 
\[
\Phi \colon \mr{Ho}(\K(\mr{Mor}(\mca A))) \to \ho{ \ms D(\mca A)^{\Delta^{1}} }	;  
\Phi(f)=f \mbox{ and } \Phi(\rho)=[\phi_{\rho}]. 
\]

Take an object $f \in \ho {\ms D(\mca A)^{\Delta^{1}}}$. 
Then $f$ can be regarded as an object in $\K(\mr{Mor}(\mca A))$. 
By the fibrant replacement $f^{\sharp}$ of $f$ in $\K(\mr{Mor}(\mca A))$, $\Phi(f^{\sharp})=f^{\sharp}$ is isomorphic to $f$ in $\ho{\ms D(\mca A)^{\Delta^{1}}}$. 
Hence the functor $\Phi$ is essentially surjective.

\textbf{Step 2}. 
We claim that any morphism $[\varphi] \colon \Phi(f) \to \Phi(g)$ in $\ho{\ms D(\mca A)^{\Delta^{1}}}$ is given by $\Phi ([\rho])$ for some 
$[\rho] \colon f\to g \in \mr{Ho}(\K(\mr{Mor}(\mca A)))$.  
To see this, let $\varphi=(\tau_{1}, \tau_{0}, \psi, h_{1}, h_{0})$ be a morphism $\varphi \colon f \to g$ in the infinity category $\ms D(\mca A)^{\Delta^{1}}$. 
Then $\tau_{0}  f $ and $g  \tau_{1}$ are chain homotopic since the relation $\delta (h_{0}-h_{1})=\tau_{0}  f- g  \tau_{1}$ holds. 

Take a factorization of the morphism $g$ in $\K (\mca A)$ as 
\begin{equation}\label{eq:factorization}
\xymatrix{
z	\ar@{^{(}->}[r]^{g_{\mr{tc}}}	&	\tilde z	\ar@{->>}[r]^{g_{\mr b}}	&	w
}, 
\end{equation}
where $g_{\mr{tc}}$ is a trivial cofibration and $g_{\mr b}$ is a fibration in $\K (\mca A)$. 
Since $\tau_{0}f$ and $g \tau_{1}$ are chain homotopic, 
we obtain the following commutative digram in $\K(\mca A)$ by using the cylinder object $C(x)$ of $x$:
\[
\xymatrix{
x	\ar[r]^-{i}	\ar[d]_{f}	&	C(x)	\ar[d]_{H}	\ar@{-->}[rd]|{\tilde H}&	x	\ar[l]_-{j}\ar[d]^{g_{\mr{tc}}\cdot \tau_{1}}	&  \ar@{=}[l]	x	\ar[d]^{\tau_{1}}	\\
y	\ar[r]_{\tau_{0}}	&	w			&	\tilde z	\ar[l]^{g_{\mr b}}	&	z	\ar[l]^{g_{\mr{tc}}}. 
}
\]
Note that $j$ is a trivial cofibration in $\K(\mca A)$. 
Since $g_{\mr b}$ is a fibration, there exists a morphism $\tilde H \colon C(x) \to \tilde z$ as indicated rendering the diagram commutative in $\K(\mca A)$. 
Thus there exists a morphism $h \in \Map_{\mr{Ch}(\mca A)}^{-1}(x, \tilde z)$ such that $g_{\mr b} h =h_{0}-h_{1}$. 

Put $\tilde \tau_{1}= g_{\mr {tc}} \tau_{1}+ \delta (h)$ where $\delta \colon \Map_{\mr{Ch}(\mca A)}^{-1}(x, \tilde z) \to \Map_{\mr{Ch}(\mca A)}^{0}(x, \tilde z)$. 
Then we obtain the commutative diagram in $\K(\mca A)$ which represents morphisms $f \rightarrow g_{\mr b}$ and $g_{\mr b} \leftarrow g$ in $\mr{Mor}(\K(\mca A))=\K(\mr{Mor}(\mca A))$: 
\[
\xymatrix{
x	\ar[r]^{\tilde{\tau_{1}}}\ar[d]_{f}	&	\tilde z	\ar[d]_{g_{\mr b}}	&	z	\ar[l]_{g_{\mr{tc}}}\ar[d]_{g}	\\
y	\ar[r]_{\tau_{0}}	&	w	&	\ar@{=}[l]	w. 
}
\]
Let $(g_{\mr{tc}}, \1_{w}) \colon g \to g_{\mr b}$ be the morphism indicated in the digram above. 
Since the object $g$ is fibrant and the morphism $(g_{\mr{tc}}, \1_{w})$ is a trivial cofibration, the lifting property in $\K(\mr{Mor}(\mca A))$ implies 
a morphism 
\begin{equation}\label{eq:retract}
(r_{1}, r_{0}) \colon g_{\mr b} \to g
\end{equation}
satisfying $(r_{1}, r_{0}) \circ (g_{\mr {tc}}, \1_{w})=\1_{g}$. 
Since $r_{0}=\1_{w}$, one can check that $\rho=(r_{1} \tilde \tau_{1}, \tau_{0})$ satisfies $\Phi([\rho])=[\varphi]$. 

\textbf{Step 3}. 
Let $\rho=(\rho_{1}, \rho_{0})$ be a morphism $\rho \colon f \to g$ for the objects 
$[f \colon x \to y]$ and $[ g \colon z \to w] \in \mr{Ho}(\K(\mr{Mor}(\mca A)))$. 
We are going to show that $\rho$ is null homotopic if $\Phi([\rho])$ is zero. 
The procedure is to find a chain homotpy, that is a pair $(h_{1}, h_{0} ) \in \Map_{\mr{Ch}(\mca A)}^{-1}(x,z) \times \Map_{\mr{Ch(\mca A)}}^{-1}(y,w)$ satisfying 
$\delta (h_{1})=\rho_{1}$, $\delta (h_{0})=\rho_{0}$, and $g  h_{1}=h_{0}  f$.


Unwinding the definitions of $\phi_{\rho}$ and of $\ho{\ms D(\mca A)^{\Delta^{1}}}$, 
there exist morphisms 
$h_{\overline{012}} \in \Map_{\mr{Ch}(\mca A)}^{-1}(x,z)$, 
$h_{\underline{012}} \in \Map_{\mr{Ch}(\mca A)}^{-1}(y,w)$, and 
$h_{\overline{0}**\underline{2}} \in \Map_{\mr{Ch}(\mca A)}^{-2}(x,w)=\Map_{\mr{Ch}(\mca A)}^{-1}(x,w[-1]) $ satisfying 
$\delta (h_{\overline{012}})= \rho_{1}$, $\delta (h_{\underline{012}})=\rho_{0}$, and 
$\delta (h_{\overline{0}**\underline{2}})= h_{\underline{012}}  f - g   h_{\overline{012}}$. 
Put $I=h_{\underline{012}}  f - g   h_{\overline{012}}$. 
Then the morphism $I$ is in $Z^{-1}_{\mr{Ch}(\mca A)}(x,w) = Z^{0}_{\mr{Ch}(\mca A)}(x,w[-1]) $ and is null homotopic by $h_{\overline{0}**\underline{2}}$. 
By the factorization of $g$ as in (\ref{eq:factorization}), we have the following commutative diagram in $\K(\mca A)$:  
\[
\xymatrix{
x	\ar[r]^-{i}	\ar[rd]_{I}	&	C(x)	\ar[d]	&	x	\ar[l]_-{j}\ar[d]^{0}	&  \ar@{=}[l]	x	\ar[d]^{0}	\\
&	w[-1]			&	\tilde z[-1]	\ar[l]^{g_{\mr b}[-1]}	&	z[-1]	\ar[l]^{g_{\mr{tc}}[-1]}. 
}
\]
Similarly as in Step 2, since $g_{\mr b}[-1]$ is a fibration, 
there exists a morphism $\tilde h \in \mr{Map}_{\mr{Ch}(\mca A)}^{-1}(x,\tilde z[-1])=\mr{Map}_{\mr{Ch}(\mca A)}^{-2}(x, \tilde z)$ such that $g_{\mr b}  \tilde h=h_{\overline{0}**\underline{2}}$.

The pair $(g_{\mr{tc}}  h_{\overline{012}} + \delta (\tilde h), h_{\underline{012}})$ satisfies $g_{\mr b}\left(	g_{\mr{tc}}  h_{\overline{012}} + \delta (\tilde h)	\right) = h_{\underline{012}} f$. 
Using the morphism $r_{1}$ in (\ref{eq:retract}), put $h_{1}=  h_{\overline{012}} + r_{1}  \delta (\tilde h)$ and $h_{0}=h_{\underline{012}}$. 
Then the pair $(h_{1}, h_{0})$ gives a desired chain homotopy for $\rho \colon f\to g$. 
\end{proof}

\begin{cor}\label{pr:MMR}
Suppose $\mca B$ is the abelian category $\mb{coh}(X)$ of coherent sheaves on a Noetherian scheme $X$. 
The bounded derived category $\mb D^{b}\left(\mr{Mor}\left(\mca B \right) \right)$ of $\mr{Mor}(\mca B)$ is equivalent to the homotopy category $\ho{\ms D_{\mr{coh}}^{b}(X)^{\Delta^{1}}}$. 
\end{cor}

\begin{proof}
Let $\mca A$ be the abelian category $\mb{Qcoh}(X)$ of quasi-coherent sheaves on $X$ and let 
\[
\Phi \colon \mr{Ho}(\K(\mr{Mor}(\mca A))) \to \ho{\ms D(\mca A)^{\Delta^{1}}}
\]
be the functor constructed in the proof of Proposition \ref{pr:eqF}. 

Define $\mb D^{b}_{\mca B}(\mr{Mor}(\mca A))$ to be the full subcategory of $\mb D(\mr{Mor}(\mca A)) = \mr{Ho}(\K(\mr{Mor}(\mca A)))$ consisting of the bounded complexes with coherent cohomologies: 
\[
\mb D^{b}_{\mca B}(\mr{Mor}(\mca A)) =\{	[f\colon x \to y ] \in \mb D(\mr{Mor}(\mca A)) \mid 
x \mbox{ and }y \in \ms D_{\mr{coh}}^{b}(X)	\}. 
\]

Then $\mb D^{b}_{\mca B}(\mr{Mor}(\mca A))$ is equivalent to $\ho{\ms D_{\mr{coh}}^{b}(\mca A)^{\Delta^{1}}}$ via the functor $\Phi$.
Thus it is enough to show that $\mb D_{\mca B}^{b}(\mr{Mor}(\mca A))$ is equivalent to $\mb D^{b}(\mr{Mor}(\mca B))$. 

Let $g \to f$ be an epi morphism in $\mr{Mor}(\mca A)$ such that $f \in \mr{Mor}(\mca B)$ and $g \in \mr{Mor}(\mca A)$. 
Then there exists a subobject $z \subset d_{1}g$ and $w \subset d_{0}g$ such that the composites $z \to d_{1}g \to d_{1}f$ and $w \to d_{0}g\to d_{0}f$ are epi and $z$ and $w$ are in $\mca B$. 
Let $\tilde w$ be the subobject of $d_{0}g$ generated by $\im (g|_{z})$ and $w$. 
Then a morphism $h \colon z \stackrel{g}{\to} \tilde w$ is a subobject of $g$ in $\mr{Mor}(\mca B)$, 
and the composite $h \to g \to f$ is also an epi morphism. 
Hence we see that the natural functor $\mb D^{b}(\mr{Mor}(\mca B)) \to \mb D^{b}_{\mca B}(\mr{Mor}(A))$ is essentially surjective. 
Then the functor gives an equivalence by \cite[Theorem B]{MR3989131}. 
\end{proof}

\begin{rmk}\label{rmk:relevant}
Keep the notation as in Corollary \ref{pr:MMR}.  
\begin{enumerate}
\item Suppose that $X$ is a smooth projective curve $C$ over $\bb C$. 
Then $\mb D^{b}(\mr{Mor}(\mca B))$ is the same as the category $\mca T_{C}$ discussed in \cite{2019arXiv190504240M}. 
Hence our category $\ho{\ms D_{\mr{coh}}^{b}(C)^{\Delta^{1}}}$ is equivalent to $\mca T_{C}$. 
Consequently, Proposition \ref{pr:BK} gives an answer to \cite[Conjecture 3.17]{2019arXiv190504240M} related with the Serre functor on $\mca T_{C}$. 
\item 
If $X$ is the affine scheme $\Spec \mb k$ of a field $\mb k$, then $\mr{Mor}(\mca B)$ is nothing but the abelian category $\mb{mod} (\bullet \to \bullet )$ of finite dimensional representations of the $A_{2}$-quiver $\bullet \to \bullet $. 
Hence our category $\ho{\ms D_{\mr{coh}}^{b}(\Spec \mb k)^{\Delta^{1}}}$ is equivalent to the bounded derived category $\mb D^{b}(\bullet \to \bullet)$ of the abelian category $\mb{mod} (\bullet \to \bullet )$. 
Thus the homotopy category $\ho{\ms D_{\mr{coh}}^{b}(X)^{\Delta^{1}}}$ is a generalization of the bounded derived category $\mb D^{b}(\bullet \to \bullet)$ of quiver representations. 
\item From a geometrical view, a quiver is nothing but a simplicial set $K$ whose arbitrary $n$-simplices degenerate for $n>1$. 
Let $\ms D_{\mr{coh}}^{b}(X)^{K}=\Fun{K}{\ms D_{\mr{coh}}^{b}(X)}$ be the infinity category of maps from any simplicial set $K$ to the infinity category $\ms D_{\mr{coh}}^{b}(X)$. 
Then $\ms D_{\mr{coh}}^{b}(X)^{K}$ can be regarded as a further generalization of the bounded derived category of quiver representations. 
It might be interesting to study $\ms D_{\mr{coh}}^{b}(X)^{K}$ 
from a perspective of the representation theory. 
\end{enumerate}
\end{rmk}


\end{document}